\documentclass[reqno]{amsart}

\usepackage[T1]{fontenc}
\usepackage[a4paper]{geometry}
\usepackage{amsmath, amssymb, mathtools, tikz-cd, amsthm, csquotes, multirow}
\usepackage{quiver}

\usepackage[hidelinks]{hyperref}

\title{Groupoidal and truncated $n$-quasi-categories}
\author{Victor Brittes}
\date{}

\numberwithin{equation}{section}

\newtheorem{theorem}[equation]{Theorem}
\newtheorem{proposition}[equation]{Proposition} 
\newtheorem{corollary}[equation]{Corollary} 
\newtheorem{lemma}[equation]{Lemma} %

\theoremstyle{definition}
\newtheorem{parag}[equation]{}
\newtheorem{definition}[equation]{Definition}

\theoremstyle{remark}
\newtheorem{example}[equation]{Example}
\newtheorem{remark}[equation]{Remark}

\def\a{\mathcal{A}}
\def\b{\mathcal{B}}
\def\cc{\mathcal{C}}
\def\d{\mathcal{D}}

\def\m{\mathcal{M}}
\def\n{\mathcal{N}}
\def\i{\mathcal{I}}
\def\j{\mathcal{J}}
\def\p{\mathcal{P}}

\def\y{\mathcal{Y}}
\def\R{\mathbb{R}}
\def\L{\mathbb{L}}
\def\W{\mathsf{W}}
\def\Sph{\mathbb{S}}
\def\Fib{\mathsf{Fib}}
\def\Cof{\mathsf{Cof}}

\mathchardef\mhyphen="2D

\DeclareMathOperator{\op}{op}

\DeclareMathOperator{\id}{id}
\DeclareMathOperator{\colim}{colim}

\DeclareMathOperator{\Hom}{Hom}

\DeclareMathOperator{\ho}{ho}
\DeclareMathOperator{\ev}{ev}
\DeclareMathOperator{\ct}{ct}
\DeclareMathOperator{\Ho}{Ho}
\DeclareMathOperator{\loc}{loc}
\DeclareMathOperator{\Set}{\mathbf{Set}}
\DeclareMathOperator{\Top}{\mathbf{Top}}

\DeclareMathOperator{\Cat}{\mathbf{Cat}}
\DeclareMathOperator{\ncat}{n \mathbf{\mhyphen Cat}}
\DeclareMathOperator{\Gpd}{\mathbf{Gpd}}

\DeclareMathOperator{\qcat}{QCat}
\DeclareMathOperator{\sset}{\mathbf{PSh}(\Delta)}
\DeclareMathOperator{\tset}{\mathbf{PSh}(\Theta_n)}
\DeclareMathOperator{\stset}{\mathbf{PSh}_{\Delta}(\Theta_n)}
\DeclareMathOperator{\psh}{\mathbf{PSh}}
\DeclareMathOperator{\spsh}{\mathbf{PSh}_{\Delta}}

\DeclareMathOperator{\kancx}{Sp}

\DeclareMathOperator{\kancxnk}{Sp_{n+k}}
\DeclareMathOperator{\kancxm}{Sp_{m}}
\DeclareMathOperator{\nqcat}{n \mhyphen QCat}
\DeclareMathOperator{\npqcat}{(n+1) \mhyphen QCat}

\DeclareMathOperator{\nmqcat}{(n-1) \mhyphen QCat}
\DeclareMathOperator{\nmqcatgp}{(n-1) \mhyphen QGpd}
\DeclareMathOperator{\nqcatgp}{n \mhyphen QGpd}

\DeclareMathOperator{\nqcatkt}{n \mhyphen QCat_{k}}

\DeclareMathOperator{\nqcatgpkt}{n \mhyphen QGpd_{k}}
\DeclareMathOperator{\thetansp}{\Theta_n Sp}
\DeclareMathOperator{\thetanspgp}{\Theta_n Gpd}

\DeclareMathOperator{\thetanspkt}{\Theta_n Sp_k}

\DeclareMathOperator{\thetanspgpkt}{\Theta_n Gpd_{k}}

\subjclass[2020]{18N20, 18N40, 18N55, 18N65, 55P15}

\keywords{n-quasi-categories, homotopy n-types, $\Theta_n$-sets}

\date{03 June 2022}

\begin{document}
\maketitle

\begin{abstract}

We define groupoidal and $(n+k)$-truncated $n$-quasi-categories, which are the translation to the world of $n$-quasi-categories of groupoidal and truncated $(\infty, n)$-$\Theta$-spaces defined by Rezk. We show that these objects are the fibrant objects of model structures on the category of presheaves on $\Theta_n$ obtained by localisation of Ara's model structure for $n$-quasi-categories. Furthermore, we prove that the inclusion $\Delta \to \Theta_n$ induces a Quillen equivalence between the model structure for groupoidal (resp. and $n$-truncated) $n$-quasi-categories and the Kan-Quillen model structure for spaces (resp. homotopy $n$-types) on simplicial sets. To get to these results, we also construct a cylinder object for $n$-quasi-categories.
\end{abstract}

\tableofcontents

\section*{Introduction}


There are several models for $(\infty, n)$-categories (see \cite{bergner2020survey} for a survey). Two closely related models are $(\infty, n)$-$\Theta$-spaces and $n$-quasi-categories, defined in \cite{rezk2010cartesian} and \cite{ara2014higher}, respectively. Ara presents in \cite{ara2014higher} two different Quillen equivalences
    $$p^*: \thetansp \rightleftarrows \nqcat: i_0^*$$
and 
    $$t_!: \nqcat \rightleftarrows \thetansp : t^! $$
between both models.

One could say that Rezk's model is more homotopical or topological, since the objects of the underlying category are simplicial presheaves, and many of the constructions and definitions of this model use the Kan-Quillen model structure on simplicial sets. Indeed, the family of theories of $(\infty, n)$-$\Theta$-spaces starts at $n = 0$, which corresponds exactly to the classical homotopy theory of simplicial sets. On the other hand, Ara's model is more combinatorial, since it deals with presheaves of sets over a given category. The family of theories of $n$-quasi-categories starts at $n=1$, where we find the theory of (1-)quasi-categories. Quasi-categories are one of the main models of $(\infty, 1)$-categories, successfully and extensively developed in \cite{lurie2009higher} and \cite{joyal2008thetheory}, for example. Thus, it is expected that $n$-quasi-categories can provide a useful model for $(\infty, n)$-categories as well. 

The main goal of this work is to show how some ideas developed by Rezk in the context of $(\infty, n)$-$\Theta$-spaces, such as that of truncated and groupoidal objects, can be transferred to the world of $n$-quasi-categories. The interest of having such definitions for $n$-quasi-categories is that, due to its more set-theorical and minimal nature, it can be easier to construct comparisons between them and more algebraic models of higher categories. For example, we expect to be able to compare $n$-truncated $n$-quasi-categories with some notion of semi-strict $n$-category, like what is done for $n=1$ and $n=2$ in \cite{campbell2020truncated} and \cite{campbell2020homotopy}, respectively, and later compare $n$-truncated $n$-quasi-groupoids with semi-strict $n$-groupoids, like what is done in \cite{brittes2022groupoidal} in the case $n=2$. Of course, for $n =1,2$, semi-strict $n$-categories are just strict $n$-categories, but for $n \geq 3$ it is known that the strict versions are not enough.

Let us now explain more precisely what are Rezk's results that we aim to transfer to Ara's model. Firstly, it is shown in \cite{rezk2010cartesian} that there is a Bousfield localisation $\thetanspkt$ of $\thetansp$ whose fibrant objects are $(n+k, n)$-$\Theta$-spaces (a truncated version of $(\infty, n)$-$\Theta$-spaces). We define $(n+k)$-truncated $n$-quasi-categories, generalizing the inductive definition given in \cite{campbell2020homotopy} for the case $n = 2, k= 0$ (the case $n = 1$ was defined in \cite{joyal2008notes} and further studied in \cite{campbell2020truncated}). We show that $(n+k)$-truncated $n$-categories are fibrant objects for a localisation of $\nqcat$, and that both Ara's Quillen equivalences presented above descend into the level of $(n+k)$-truncated objects.

Also in \cite{rezk2010cartesian}, there is another localisation $\thetanspgp$ of $\thetansp$, whose homotopy category is equivalent to the homotopy category of spaces. We define groupoidal $n$-quasi-categories, generalising the definition given in \cite{brittes2022groupoidal} for the case $n = 2$. We show that these are the fibrant objects of a localisation $\nqcatgp$ of $\nqcat$, and that $\nqcatgp$ and $\thetanspgp$ are Quillen equivalent via both adjunctions above.

A last result of \cite{rezk2010cartesian} provides a direct equivalence comparing $\thetanspgp$ and $\kancx$ (the Kan-Quillen model category on simplicial sets), and similarly for their truncated versions. Therefore, we could use the results of the previous paragraphs to provide comparisons between $\nqcatgp$ and $\kancx$. However, the inclusion $i: \Delta \to \Theta_n$ allows us to also provide a different, new Quillen equivalence between $\nqcatgp$ and $\kancx$, which does not factor through $\thetanspgp$. Moreover, we show that this Quillen equivalence induces another one between $\nqcatgpkt$ and $\kancxnk$ (the model structure for homotopy $(n+k)$-types). 

\vspace{2mm}

We would like to highlight the main differences in the techniques used in this work and in \cite{brittes2022groupoidal}, which addressed similar questions for the case $n=2$. Both repeatedly use the theory of localisation of model structures. However, some results for $n=2$, established in \cite{campbell2020homotopy} and used in \cite{brittes2022groupoidal}, do not easily generalize to the general case. We give 3 examples. 

The first one concerns the way we show that $n$-quasi-groupoids are equivalent to spaces via an adjunction induced by the inclusion $i: \Delta \to \Theta_n$ (Theorem \ref{quillen eq kan cx groupoidal nqcat} of this article for general $n$ and \cite[Theorem 3.14]{brittes2022groupoidal} for $n=2$) . In the case of 2-quasi-groupoids, we use a Quillen equivalence proven by Campbell between quasi-categories and 2-quasi-categories which are locally Kan complexes, and then proceed by localisation. Here, the strategy is to fit the adjunction we want to show is a Quillen equivalence in a triangle of Quillen adjunctions, where we use Rezk's $(\infty, n)$-$\Theta$-groupoids to construct one of its edges. 

Secondly, the comparison between 2-truncated 2-quasi-groupoids and homotopy 2-types \cite[Corollary 6.3]{brittes2022groupoidal} uses Campbell's nerve for 2-categories \cite{campbell2020homotopy} and Moerdijk-Svensson's equivalence between 2-groupoids and homotopy 2-types \cite{moerdijk1993algebraic}. To compare $(n+k)$-truncated $n$-quasi-groupoids and homotopy $(n+k)$-types (Theorem \ref{quillen eq truncated n-q-groupoids and homotopy n-types}), we show that two (a priori different) model structures are the same: the one for $(n+k)$-truncated $n$-quasi-groupoids ($\nqcatgpkt$) and a localisation of $\nqcatgp$, deliberately built to be equivalent to the one for $(n+k)$-homotopy types. The key idea to do so is to compare the boundaries of representable $\Theta_n$-sets, which can be interpreted in the model structure $\nqcatgp$ as model for spheres. In the Kan-Quillen model structure on simplicial sets, for every $k > 0$, we can think of $\partial \Delta[k]$ as a model for the sphere $\mathbb{S}^{k-1}$, which can be made precise by using the realisation functor $\sset \to \Top$, or by giving an abstract definition of sphere as a given homotopy pushout, which is the approach we use here. When working with presheaves on $\Theta_n$, there can be more than one representable such that its boundary models a given sphere. For example, in $\Theta_2$, the boundaries of both $[2]([0], [0])$ and $[1]([1])$ (respectively represented below) are models for $\mathbb{S}^1$. 

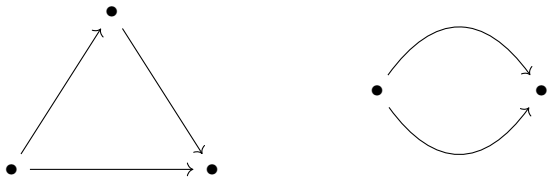
\[\begin{tikzcd}
	& \bullet \\
	&&&& \bullet && \bullet \\
	\bullet && \bullet
	\arrow[from=3-1, to=1-2]
	\arrow[from=1-2, to=3-3]
	\arrow[from=3-1, to=3-3]
	\arrow[curve={height=-30pt}, from=2-5, to=2-7]
	\arrow[curve={height=30pt}, from=2-5, to=2-7]
\end{tikzcd}\]

The third difference is in the proof that the suspension functor $\Sigma: \psh(\Theta_n) \to * \sqcup * / \psh(\Theta_{n+1})$ is a left Quillen functor. When $n=1$, this is proven in \cite{campbell2020homotopy}, using a simple recognition result for left Quillen functors from the model category for quasi-categories. This recognition result generalises to $n$-quasi-categories (Proposition \ref{recognition principle for left quillen functors}), with the addition of one extra condition to be verified, involving a certain cylinder object. The problem is that we do not know a simple combinatorial description for the "obvious" choice of cylinder for $n$-quasi-categories (the product with an interval object). The solution we found is to build a smaller, more combinatorial cylinder, in order to use such recognition principle. This construction, therefore fundamental to the proof of this result, is also of independent interest for other applications in the theory of $n$-quasi-categories. With this new cylinder, we are able to get a refined version of Proposition \ref{recognition principle for left quillen functors}, stated in Theorem \ref{recognition principle for left quillen functors 2}.

\vspace{2mm}

The following table summarizes all the model categories we work with in this article, the respective fibrant-cofibrant objects, and how each one of them fits in the general framework of $(r, s)$-categories: higher (weak) categories where all $k$-morphisms are (weakly) invertible for $k > s$ and (weakly) trivial for $k > r$.

\begin{center}
    \begin{tabular}{|c|c|c|c|}
  \hline
  \multicolumn{1}{|c|}{Category} & Model structure & Fibrant objects & $(r,s)$-category \\
  \hline
  \multirow{3}{*}{$\sset$} & $\kancx$ & Kan complexes & $(\infty, 0)$ \\
  & $\kancxm$ & Homotopy $m$-types & $(m,0)$ \\
  \hline
  \multirow{3}{*}{$\tset$} & $\nqcat$ & $n$-quasi-categories & $(\infty, n)$ \\
  & $\nqcatgp$ & $n$-quasi-groupoids  & $(\infty, 0)$ \\
  & $\nqcatkt$ & $(n+k)$-truncated $n$-quasi-categories  & $(n+k, n)$ \\
  & $\nqcatgpkt$ & $(n+k)$-truncated $n$-quasi-groupoids  & $(n+k, 0)$ \\
  \hline
  \multirow{3}{*}{$\stset$} & $\thetansp$ & $(\infty, n)$-$\Theta$-categories & $(\infty, n)$ \\
  & $\thetanspgp$ & $(\infty, n)$-$\Theta$-groupoids & $(\infty, 0)$ \\
  & $\thetanspkt$ & $(n+k,n)$-$\Theta$-categories & $(n+k, n)$ \\
  & $\thetanspgpkt$ & $(n+k,n)$-$\Theta$-groupoids & $(n+k, 0)$ \\
  \hline
\end{tabular}
\end{center}

\vspace{2mm}

\textbf{Organisation of the paper.} In the first preliminary section, we recall some aspects of the theory of localisation of model structures, our main technical tool. In section 2, we briefly recall the definition of the category $\Theta_n$, Ara's and Rezk's model structures for $(\infty, n)$-categories, and Ara's Quillen equivalences between them. The third is devoted to the construction of an alternative cylinder object for $n$-quasi-categories, that will be used in the next section to show that the suspension functor between $\Theta_n$-sets and bipointed $\Theta_{n+1}$-sets is left Quillen. Its right adjoint, the functor $\Hom$, plays a fundamental role in the definitions of truncated $n$-quasi-categories, presented in section 4, and of groupoidal $n$-quasi-categories, in section 5. Besides the definitions, these two sections also contain theorems stating the existence, for each version of $n$-quasi-category (truncated and groupoidal), of model structures which are Quillen equivalent to Rezk's ones. The last section provides a third, direct, comparison between $n$-quasi-groupoids and spaces, and its truncated version, which uses the sphere argument mentioned above

\vspace{2mm}

\textbf{Acknowledgement.} We would like to thank Muriel Livernet and Clemens Berger for the helpful and insightful conversations. We would also like to thank Felix Loubaton for the idea for the construction of the cylinder of Section 3.

\vspace{2mm}

\textbf{Notation.} If $\a$ is a small category, we denote by $\widehat{\a}$ the category of presheaves of sets on $\a$, i.e., of functors $\a^{\op} \to \Set$. If $X$ is an object of $\widehat{\a}$ and $a$ is an object of $\a$, we write $X_a$ for the set $X(a)$. The Yoneda embedding $a \mapsto \a[a] := \Hom_{\a}(-, a)$ defines a fully faithful functor $\a \to \widehat{\a}$. If $f: a \to b$ is a morphism in $\a$, we still denote by $f: \a[a] \to \a[b]$ its image under the Yoneda embedding. Sometimes, when there is no risk of confusion, we also denote by $a$ the image of an object $a \in \a$ by the Yoneda embedding.

When representing an adjunction by $F: \cc \rightleftarrows \d: G$, the functor $F$ is left adjoint to the functor $G$.

If $\a$ and $\b$ are two small categories and $\varphi: \a \to \b$ is a functor, we denote by $\varphi^*: \widehat{\b} \to \widehat{\a}$ the \textit{restriction functor} given by precomposition with (the opposite of) $\varphi$. Given an adjunction  $\varphi: \a \rightleftarrows \b: \psi$, considering restriction functors induces an adjunction $\varphi^*: \widehat{\b} \rightleftarrows \widehat{\a}: \psi^*$ between presheaf categories.

If $\a$ is a small category and $\cc$ is a cocomplete category, a functor $f: \a \to \cc$ induces an adjunction that we will denote by $f_!: \widehat{\a} \rightleftarrows \cc: f^!$, where $f_!$ is the left Kan extension of $f$ along the Yoneda embedding, and $f^!$ is the \textit{nerve functor} defined by $f^!(X)_a = \Hom_{\cc}(f(a), X)$ for $X \in \cc$ and $a \in \a$. In the case where $\cc = \widehat{\b}$ for some small category $\b$ and $f: \a \to \widehat{\b}$ is the composite of a functor $\varphi: \a \to \b$ with the Yoneda embedding, we abuse notation and denote by $\varphi_!: \widehat{\a} \rightleftarrows \widehat{\b}: \varphi^!$ the induced adjunction. Note that the Yoneda lemma implies that $\varphi^!$ is isomorphic to the restriction functor $\varphi^*$.

\section{Localisation of model category structures}

We recall some notions and results about the localisation of model categories. A complete reference is \cite{hirschhorn2003model}.

\begin{parag} \label{bousfield localisation}
Let $(\m, \Cof, \W, \Fib)$ be a model category structure on a category $\m$. A model category structure $(\m, \Cof_{\loc}, \W_{\loc}, \Fib_{\loc})$ on $\m$ is a \textit{(left) Bousfield localisation} of $(\m, \Cof, \W, \Fib)$ if $\Cof_{\loc} = \Cof$ and $\W \subset \W_{\loc}$. When studying a model category structure and a given localisation, we shall write $\m$ and $\m_{\loc}$ to refer to the original model structure and to its localisation, respectively. We shall also call \textit{local fibration} (resp. \textit{local fibrant object}, resp. \textit{local weak equivalence}) a fibration (resp. fibrant object, resp. weak equivalence) of $\m_{\loc}$. Note that we have a Quillen adjunction
$$\id: \m \rightleftarrows \m_{\loc}: \id$$

We see that a Bousfield localisation is completely determined by its fibrant objects, i.e., the local fibrant objects. It is useful to know that a morphism between local fibrant objects is a weak equivalence (resp. fibration) in $\m$ if and only if it is a local weak equivalence (resp. local fibration).
\end{parag}

\begin{parag} \label{homotopy mapping space and local objects}
Given a model category $\m$ and two objects $X, Y$ of $\m$, we can consider the \textit{homotopy mapping space} $\underline{\Ho \m}(X, Y)$, which is the image of the pair $(X, Y)$ by the functor $\underline{\Ho \m}: \Ho(\m)^{\op} \times \Ho(\m) \to \Ho(\widehat{\Delta})$\footnote{When writing $\Ho(\widehat{\Delta})$, we always consider the Kan-Quillen model structure on the category of simplicial sets} induced by $\Hom_{\m}: \m^{\op} \times \m \to \Set$ (see \cite[A.2]{ara2014higher} or \cite{hirschhorn2003model} for the details). 

Let $S$ be a class of morphisms of $\m$. An object $X$ of $\m$ is \textit{$S$-local} (or \textit{local with respect to $S$}) if for every morphism $f: A \to B$ of $S$, the induced map
$$\underline{\Ho \m}(f, X): \underline{\Ho \m}(B, X) \to \underline{\Ho \m}(A, X)$$
is an isomorphism (in the homotopy category $\Ho(\widehat{\Delta})$).

A morphism $f: A \to B$ of $\m$ is an \textit{$S$-equivalence} if for every $S$-local object $X$, the induced map
$$\underline{\Ho \m}(f, X): \underline{\Ho \m}(B, X) \to \underline{\Ho \m}(A, X)$$
is an isomorphism (in the homotopy category $\Ho(\widehat{\Delta})$).

If there is a Bousfield localisation $\m_{\loc}$ of $\m$ whose local fibrant objects are the $S$-local objects and whose weak equivalences are the $S$-equivalences, we say that $\m_{\loc}$ is a \textit{(Bousfield) localisation of $\m$ with respect to $S$}, and we denote it by $L_S \m$.
\end{parag}

\begin{example} \label{Kan-Quillen vs Joyal}
The Kan-Quillen model category structure on simplicial sets is a localisation of Joyal's model structure with respect to the morphism $\Delta[1] \to \Delta[0]$, cf. \cite[Proposition 3.30]{campbell2020truncated}.
\end{example}

We will state a result, due to Smith, about the existence of the localisation of a model category with respect to a certain set of morphisms. Before, let us recall some definitions. A model category is \textit{left proper} if the pushout of every weak equivalence along a cofibration is a weak equivalence. A model category where all objects are cofibrant is left proper (see \cite[Corollary 13.1.3]{hirschhorn2003model}). A model category is \textit{combinatorial} if it is cofibrantly generated and locally presentable. 


\begin{theorem} \label{smith existence theorem}
Let $\m$ be a left proper and combinatorial model category. Let $S$ be a set of morphisms of $\m$. Then the localisation $L_S \m$ of $\m$ with respect to $S$ exists and is left proper and combinatorial.
\end{theorem}

\begin{proof}
See \cite[Theorem 4.7]{barwick2010left}.
\end{proof}

\begin{remark}
If $F: \m \rightleftarrows \n: G$ is a Quillen adjunction between model categories, the induced adjunction between the homotopy categories is usually denoted by $\L F: \Ho (\m) \rightleftarrows \Ho(\n): \R G$. In what follows, we will abuse language and also denote by $\L F$ the functor $\L F := F Q: \m \to \n$, where $Q$ is a fixed functorial cofibrant replacement in the model category $\m$. In all the applications presented in this paper, all objects of $\m$ will be cofibrant, and so we shall take $Q = \id_{\m}$.
\end{remark}

We can transfer localisations of model structures along Quillen adjunctions.

\begin{proposition} \label{local fibrant objects and localisation}
Let $F: \m \rightleftarrows \n: G$ be a Quillen adjunction between model categories $\m$ and $\n$. Let $S$ be a class of morphisms of $\m$. A fibrant object $Y$ of $\n$ is $\L F(S)$-local if and only if $G(Y)$ is $S$-local. 
\end{proposition}

\begin{proof}
See \cite[Proposition 3.1.12]{hirschhorn2003model}.
\end{proof}

\begin{theorem} \label{transfer of localisation}
Let $F: \m \rightleftarrows \n: G$ be a Quillen adjunction between model categories $\m$ and $\n$. Let $S$ be a class of morphisms of $\m$. If the localisations $L_S \m$ and $L_{\L F(S)} \n$ exist, then 
$$F: L_S \m \rightleftarrows L_{\L F(S)} \n: G$$
is a Quillen adjunction between the localised model categories. 

Moreover, if $F: \m \rightleftarrows \n: G$ is a Quillen equivalence, then so is $F: L_S \m \rightleftarrows L_{\L F(S)} \n: G$.
\end{theorem}

\begin{proof}
See \cite[Proposition 3.3.20]{hirschhorn2003model}
\end{proof}

The two following theorems (recalled from the appendix of \cite{campbell2020truncated}) present ways to know when Quillen adjunctions and equivalences are preserved after localisation.

\begin{theorem} \label{criterion for quillen adjunction after localisation}
    Let $F: \m \rightleftarrows \n: G$ be a Quillen adjunction between model categories $\m$ and $\n$. Let $\m_{\loc}$ be a localisation of $\m$. The adjunction $F: \m_{\loc} \rightleftarrows \n: G$ is Quillen if and only if for every fibrant every object $Y$ of $\n$, its image $G(Y)$ is fibrant in $\m_{\loc}$.
\end{theorem}

\begin{proof}
    See \cite[Proposition A.13]{campbell2020truncated}.
\end{proof}

\begin{theorem} \label{recognition of equivalence after localisastion}
    Let $F: \m \rightleftarrows \n: G$ be a Quillen equivalence between model categories $\m$ and $\n$. Let $\m_{\loc}$ and $\n_{\loc}$ be localisations of these model structures. We have a Quillen equivalence $F: \m_{\loc} \rightleftarrows \n_{\loc}: G$ if and only if a fibrant object $Y$ of $\n$ is fibrant in $\n_{\loc}$ precisely when $G(Y)$ is fibrant in $\m_{\loc}$.
\end{theorem}

\begin{proof}
    See \cite[Theorem A.15]{campbell2020truncated}.
\end{proof}

Next, we state a proposition which will allow us to understand successive localisations of a model category.

\begin{lemma} \label{lemma successive localisations}
Let $\m$ be a model category and $\m_{\loc}$ be a Bousfield localisation of $\m$. For every object $X$ and every local fibrant object $Y$ of $\m$, the homotopy mapping spaces $\underline{\Ho \m}(X, Y)$ and $\underline{\Ho \m_{\loc}}(X, Y)$ are naturally isomorphic in $\Ho(\widehat{\Delta})$.
\end{lemma}

\begin{proof}
See \cite[Lemma A.4]{ara2014higher}
\end{proof}

\begin{proposition} \label{proposition succesive localisations}
Let $\m$ be a model category and $S, T$ be two classes of morphisms of $\m$. Suppose that the localisations $L_S \m$, $L_T \m$, $L_T L_S \m$, $L_S L_T \m$ and $L_{S \cup T} \m$ exist. Then the model categories $L_T L_S \m$, $L_S L_T \m$ and $L_{S \cup T} \m$ are the same.
\end{proposition}

\begin{proof}
Since a model structure is completely determined by its cofibrations and fibrant objects (cf. \cite[Proposition E.1.10]{joyal2008thetheory}) and the 3 considered model structures have the same cofibrations, it is sufficient to show that they have the same fibrant objects. Let $X$ be an object of $\m$. We claim that the following assertions are equivalent:
\begin{enumerate}
    \item $X$ is a $T$-local object of $L_S \m$
    \item $X$ is an $S$-local object of $L_T \m$
    \item $X$ is an $(S \cup T)$-local objects of $\m$
\end{enumerate}

We will show that $(1) \Leftrightarrow (3)$. The equivalence $(2) \Leftrightarrow (3)$ follows by exchanging the roles of $S$ and $T$.

$(1) \Rightarrow (3)$ Suppose that $X$ is a $T$-local object of $L_S \m$. We have to show that, for every $f \in S \cup T$, $f: A \to B$, the map
$$\underline{\Ho \m}(f, X): \underline{\Ho \m}(B, X) \to \underline{\Ho \m}(A, X)$$
is an isomorphism. This is true if $f \in S$, since $X$ is fibrant in $L_T L_S \m$, so it is in particular fibrant in $L_S \m$, which means it is $S$-local in $\m$. If $f \in T$, we consider the commutative square
\[\begin{tikzcd}
	{\underline{\text{Ho} \mathcal{M}}(B, X)} && {\underline{\text{Ho} \mathcal{M}}(A, X)} \\
	{\underline{\text{Ho} L_S\mathcal{M}}(B, X)} && {\underline{\text{Ho} L_S\mathcal{M}}(A, X)}
	\arrow["{\underline{\text{Ho} \mathcal{M}}(f, X)}", from=1-1, to=1-3]
	\arrow["{\underline{\text{Ho} L_S\mathcal{M}}(f, X)}", from=2-1, to=2-3]
	\arrow["\cong"', from=1-1, to=2-1]
	\arrow["\cong", from=1-3, to=2-3]
\end{tikzcd}\]
where the isomorphisms are given by Lemma \ref{lemma successive localisations}. The bottom arrow is an isomorphism, since $X$ is $T$-local in $L_S \m$ by assumption, and thus the top arrow is also an isomorphism, as desired.

$(3) \Rightarrow (1)$ Let $X$ be a $(S \cup T)$-local object of $\m$. Let $f: A \to B$ be a morphism in the class $T$. We have to show that the map
$$\underline{\Ho L_S \m}(f, X): \underline{\Ho L_S \m}(B, X) \to \underline{\Ho L_S \m}(A, X)$$
is an isomorphism. Since $X$ is $(S \cup T)$-local in $\m$, it is in particular $S$-local in $\m$, and so it is a fibrant object of $L_S \m$. Therefore, we can use Lemma \ref{lemma successive localisations} to consider a commutative square as above. The top arrow is an isomorphism since $X$ is $T$-local in $\m$, which implies that the bottom arrow is also an isomorphism.
\end{proof}

We end this section recalling some results about model category structures on slice categories.

\begin{parag} \label{model structure for slice categories}
If $\m$ is a category and $C$ is an object of $\m$, we denote by $C / \m$ the \textit{slice category} of objects under $C$. The objects of $C / \m$ are pairs $(A, a)$, where $A$ is an object of $\m$ and $a: C \to A$ is a morphism of $\m$. A morphism $f: (A, a) \to (B, b)$ in $C / \m$ is a morphism $f: A \to B$ of $\m$ such that $f a = b$. There is an obvious forgetful functor $C / \m \to \m$ taking $(A, a)$ to $A$.

When $\m$ is endowed with a model category structure, there is an induced model structure on $C /\m$ where a morphism is a weak equivalence (resp. cofibration, resp. fibration) when its image by the forgetful functor is a weak equivalence (resp. cofibration, resp. fibration) of $\m$ \cite[Theorem 7.6.5]{hirschhorn2003model}. 
\end{parag}

\begin{lemma} \label{lemma local objects slice category}
Let $\m$ be a model category, $C$ be a cofibrant object of $\m$ and $f: (A, a) \to (B, b)$ be a morphism in $C / \m$. A fibrant object $X$ of $\m$ is local with respect to $f$ in $\m$ if and only if for every $x: C \to X$, the object $(X, x)$ of $C / \m $ is local with respect to $f$ in $C / \m$.
\end{lemma} 

\begin{proof}
    Exercise in \cite[§8.1]{campbell2020homotopy}.
\end{proof}


\section{Ara's $n$-quasi-categories and Rezk's $(\infty, n)$-$\Theta$-spaces}


Let $n \geq 0$. We recall the recursive definition of the category $\Theta_n$, using the wreath product introduced in \cite{berger2007iterated}. 

\begin{parag} \label{parag wreath product}
Let $\cc$ be a category. The \textit{wreath product} $\Delta \wr \cc$ is the category described as follows. The objects are lists $[p](x_1, \ldots, x_p)$, where $p \geq 0$ and $x_1, \ldots, x_p$ are objects of $\cc$. A morphism $[f](\alpha): [p](x_1, \ldots, x_p) \to [q](y_1, \ldots, y_q)$ is the data of a morphism $f: [p] \to [q]$ in $\Delta$ and of morphisms $\alpha^{i}_j: x_i \to y_j$ in $\cc$ for every $i, j$ such that $f(i-1) < j \leq f(i)$.
\end{parag}

\begin{parag} \label{parag special functors in the theta construction}    
There is a \textit{suspension} functor $\sigma: \cc \to \Delta \wr \cc$, which sends an object $x$ to  $[1](x)$ and a morphism $f: x \to y$ to $[\id_1](f)$. If the category $\cc$ has a terminal object $t$, there is an \textit{inclusion} functor $i: \Delta \to \Delta \wr \cc$, defined on objects by $[p] \mapsto [p](t, \ldots, t)$. The inclusion functor $i$ has a left adjoint, given by the \textit{truncation} $\pi: \Delta \wr \cc \to \Delta$ sending $[p](x_1, \ldots, x_p)$ to $[p]$.
\end{parag}

\begin{parag} \label{parag def thetan}
    Let $\Theta_0$ be the terminal category. We recursively define $\Theta_n$ for $n > 0$ by letting 
    $$\Theta_n := \Delta \wr \Theta_{n-1}$$ 
    We note that $\Theta_1$ is exactly $\Delta$. Moreover, by the previous paragraph we have a suspension
    $$\sigma: \Theta_n \to \Theta_{n+1}$$
    and an inclusion (since $\Theta_{n-1}$ has a terminal object $[0]$)
    $$i: \Delta \to \Theta_n$$
    for every $n \geq 0$.
\end{parag}

\begin{parag} \label{parag inclusion thetan in ncat}
    For every $n \geq 0 $, there is fully faithful inclusion $\Theta_n \to \ncat$. Conceptually, we can see $\Theta_0$ as the full subcategory of $0\mhyphen\Cat := \Set$ formed the singleton. Reasoning inductively and using the fact that the wreath product preserves fully faithfulness, we can see $\Theta_n = \Delta \wr \Theta_{n-1}$ as a full subcategory of $\Delta \wr (n-1)\mhyphen \Cat$, which is a full subcategory of $\ncat$. From now on, we may identify $\Theta_n$ with a full subcategory of $\ncat$.

    More explicitly, we picture an object $[p](\theta_1, \ldots, \theta_p)$ of $\Theta_n$ as the $n$-category freely generated by the ($(n-1)\mhyphen \Cat$)-graph with objects $0, 1, \ldots p$ and hom-$(n-1)$-categories $\Hom(i-1, i) = \theta_i$ for $1 \leq i \leq p$.

    If $\cc$ is an $n$-category, the \textit{suspension} $\sigma(\cc)$ is the $(n+1)$-category with two objects $0$ and $1$, and with hom-$n$-categories given by $\Hom(0,1) = \cc$, $\Hom(1,0) = \emptyset$, $\Hom(0,0) = \Hom(1,1) = \{*\}$. The restriction of the suspension functor $\sigma: \ncat \to (n+1)\mhyphen\Cat$ to $\Theta_n$ factors through $\Theta_{n+1}$ and is exactly the suspension defined in §\ref{parag def thetan}.
\end{parag}

\begin{example} \label{example objects of theta n}
The object $\theta = [3](2, 0, 1)$ of $\Theta_2$ is the free 2-category generated by the 2-graph
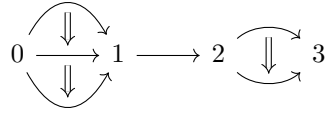
\[\begin{tikzcd}
	0 & 1 & 2 & 3
	\arrow[""{name=0, anchor=center, inner sep=0}, curve={height=-24pt}, from=1-1, to=1-2]
	\arrow[""{name=1, anchor=center, inner sep=0}, curve={height=24pt}, from=1-1, to=1-2]
	\arrow[""{name=2, anchor=center, inner sep=0}, from=1-1, to=1-2]
	\arrow[from=1-2, to=1-3]
	\arrow[""{name=3, anchor=center, inner sep=0}, curve={height=-12pt}, from=1-3, to=1-4]
	\arrow[""{name=4, anchor=center, inner sep=0}, curve={height=12pt}, from=1-3, to=1-4]
	\arrow[shorten <=3pt, shorten >=3pt, Rightarrow, from=0, to=2]
	\arrow[shorten <=3pt, shorten >=3pt, Rightarrow, from=2, to=1]
	\arrow[shorten <=3pt, shorten >=3pt, Rightarrow, from=3, to=4]
\end{tikzcd}\]

Its suspension $\sigma(\theta) \in \Theta_3$ is the free 3-category generated by the 3-graph (where the small horizontal arrows should be triple)
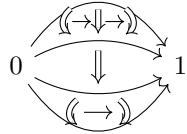
\[\begin{tikzcd}
	0 && 1
	\arrow[""{name=0, anchor=center, inner sep=0}, curve={height=-30pt}, from=1-1, to=1-3]
	\arrow[""{name=1, anchor=center, inner sep=0}, curve={height=-12pt}, from=1-1, to=1-3]
	\arrow[""{name=2, anchor=center, inner sep=0}, curve={height=12pt}, from=1-1, to=1-3]
	\arrow[""{name=3, anchor=center, inner sep=0}, curve={height=30pt}, from=1-1, to=1-3]
	\arrow[""{name=4, anchor=center, inner sep=0}, shift right=3, curve={height=6pt}, shorten <=3pt, shorten >=3pt, Rightarrow, from=0, to=1]
	\arrow[""{name=5, anchor=center, inner sep=0}, shift left=3, curve={height=-6pt}, shorten <=3pt, shorten >=3pt, Rightarrow, from=0, to=1]
	\arrow[""{name=6, anchor=center, inner sep=0}, shorten <=2pt, shorten >=2pt, Rightarrow, from=0, to=1]
	\arrow[""{name=7, anchor=center, inner sep=0}, shift right=2, curve={height=6pt}, shorten <=3pt, shorten >=3pt, Rightarrow, from=2, to=3]
	\arrow[shorten <=3pt, shorten >=3pt, Rightarrow, from=1, to=2]
	\arrow[""{name=8, anchor=center, inner sep=0}, shift left=2, curve={height=-6pt}, shorten <=3pt, shorten >=3pt, Rightarrow, from=2, to=3]
	\arrow[shorten <=2pt, shorten >=2pt, from=4, to=6]
	\arrow[shorten <=2pt, shorten >=2pt, from=6, to=5]
	\arrow[shorten <=4pt, shorten >=4pt, from=7, to=8]
\end{tikzcd}\]
\end{example}

\begin{parag} \label{parag globes}
    Of special interest among the objects of $\Theta_n$ are the \textit{globes} $D_0, \ldots, D_n$. For $n = 0$, the globe $D_0$ is the only object of $\Theta_0$. For $n > 0$, we have $D_0 = [0]$ and $D_k = \sigma(D_{k-1})$ for $1 \leq k \leq n$. 

    For $1 \leq k \leq n$, there are two morphisms $s,t: D_{k-1} \to D_k$ in $\Theta_n$ corresponding to sending the non-identity $(k-1)$-cell of $D_{k-1}$ to the source and target of the non-identity $k$-cell of $D_{k}$, respectively.

    On the other hand, the unique morphism $\tau^1_1: [1] \to [0]$ of $\Delta = \Theta_1$ induces by suspensions morphisms $\tau^n_k: D_k \to D_{k-1}$ in $\Theta_n$ for $2 \leq k \leq n$. The unique morphism $D_1 \to D_0$ in $\Theta_n$ is denoted by $\tau^n_1$. In practice, the $n$-functor $\tau^n_k: D_k \to D_{k-1}$ sends the two non-identity $(k-1)$-cells of $D_k$ to the only non-identity $(k-1)$-cell of $D_{k-1}$.
\end{parag}

\begin{example} \label{example globes}
    The globes $D_0, \ldots, D_3$ of $\Theta_3$ are pictured below
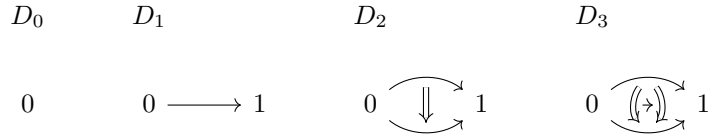
\[\begin{tikzcd}
	{D_0} & {D_1} && {D_2} && {D_3} \\
	0 & 0 & 1 & 0 & 1 & 0 & 1
	\arrow[from=2-2, to=2-3]
	\arrow[""{name=0, anchor=center, inner sep=0}, curve={height=-12pt}, from=2-4, to=2-5]
	\arrow[""{name=1, anchor=center, inner sep=0}, curve={height=12pt}, from=2-4, to=2-5]
	\arrow[""{name=2, anchor=center, inner sep=0}, curve={height=-12pt}, from=2-6, to=2-7]
	\arrow[""{name=3, anchor=center, inner sep=0}, curve={height=12pt}, from=2-6, to=2-7]
	\arrow[shorten <=3pt, shorten >=3pt, Rightarrow, from=0, to=1]
	\arrow[""{name=4, anchor=center, inner sep=0}, curve={height=6pt}, shorten <=4pt, shorten >=4pt, Rightarrow, from=2, to=3]
	\arrow[""{name=5, anchor=center, inner sep=0}, curve={height=-6pt}, shorten <=4pt, shorten >=4pt, Rightarrow, from=2, to=3]
	\arrow[shorten <=2pt, shorten >=2pt, from=4, to=5]
\end{tikzcd}\]
\end{example}


\begin{parag} \label{parag theta_n sets}
    We denote by $\tset$ the category of presheaves on $\Theta_n$, i.e., of functors $\Theta_n^{\op} \to \Set$ and natural transformations between them. The representable presheaves are denoted by $\Theta_n[\theta]$, and by definition we have 
    $$\Theta_n[\theta]_{\theta'} = \Hom_{\Theta_n}(\theta', \theta)$$
    for every $\theta, \theta' \in \Theta_n$.

    The \textit{boundary} $\partial \Theta_n[\theta]$ is the presheaf generated by the monomorphisms $\theta' \to \theta$ which are not the identity. We denote by $\delta_{\theta}: \partial \Theta_n[\theta] \to \Theta_n[\theta]$ the boundary inclusion monomorphism. 
\end{parag}

\begin{parag} \label{parag nerve}
The inclusion $\Theta_n \to \ncat$ induces a fully faithful (cf. \cite[Theorem 1.12]{berger2002cellular}) nerve functor 
$$N_n: \ncat \to \tset$$
Explicitly, if $\cc$ is an $n$-category and $\theta$ is an object of $\Theta_n$, we have
$$N_n(\cc)_{\theta} = \Hom_{\ncat}(\theta, \cc)$$
\end{parag}

\begin{parag} \label{parag nqcat}
    Ara defines in \cite{ara2014higher} a model category structure $\nqcat$ on the category $\tset$ which provides a model for $(\infty, n)$-categories. The fibrant objects of $\nqcat$ are called \textit{$n$-quasi-categories}.
\end{parag}

We will state a recognition principle for left Quillen functors from $\nqcat$, generalizing the one for the case $n=2$ given by Campbell \cite[Proposition 4.13]{campbell2020homotopy}. To this end, we introduce three classes of morphisms of presheaves on $\Theta_n$. 

\begin{parag} \label{parag spine inclusions}
    The first class, denoted by $\i_n$, is the class of \textit{spine inclusions}. For $n = 1$, $[p] \in \Theta_1 = \Delta$, the spine $I[p]$ is the colimit of the image in $\sset$ of the following diagram $\d([p])$ in $\Delta$
\[\begin{tikzcd}
	{D_1} && \ldots && {D_1} \\
	& {D_0} && {D_0}
	\arrow["t", from=2-2, to=1-1]
	\arrow["s"', from=2-2, to=1-3]
	\arrow["t", from=2-4, to=1-3]
	\arrow["s"', from=2-4, to=1-5]
\end{tikzcd}\]
where there are $p$ copies of $D_1$. There is an obvious inclusions $i_p: I[p] \to \Delta[p]$ induced by the universal property of the colimit.

For $n > 1$, define the spines inductively. Consider an object $\theta = [p](\theta_1, \ldots, \theta_p)$ of $\Theta_n$. The spine $I[\theta_i]$ associated to $\theta_i$ is defined as the colimit of (the image by the Yoneda embedding of) a diagram $\d(\theta_i)$ in $\Theta_{n-1}$. The spine $I[\theta]$ is then the colimit in $\tset$ of the diagram
\[\begin{tikzcd}
	{\sigma(\mathcal{D}(\theta_1))} && \ldots && {\sigma(\mathcal{D}(\theta_p))} \\
	& {D_0} && {D_0}
	\arrow["t", from=2-2, to=1-1]
	\arrow["s"', from=2-2, to=1-3]
	\arrow["t", from=2-4, to=1-3]
	\arrow["s"', from=2-4, to=1-5]
\end{tikzcd}\]
where the target of the $t$'s (resp. $s$'s) are the rightmost (resp. leftmost) globe appearing in $\d(\theta_i)$. Once again, there are inclusions
$$i_{\theta}: I[\theta] \to \Theta_n[\theta]$$

The class $\i_n$ is formed by all the spine inclusions, i.e., 
$$\i_n := \{i_\theta: I[\theta] \to \Theta_n[\theta], \theta \in \Theta_n\}$$
\end{parag}

\begin{parag} \label{parag equivalences}
The second class is the one of \textit{generating equivalences}. Let $J = J_1$ be the free groupoid generated by one arrow $0 \to 1$. We define the $k$-category $J_k$ as the suspension $\sigma(J_{k-1})$ for $k > 1$. The functor $j_1: J \to D_0$ induces functors
$$j_k : J_k \to D_{k-1}$$ for every $k \geq 1$. For $n \geq 1$, the class $\j_n$ is given by
$$\j_n := \{N_n(j_k): N_n(J_k) \to N_n(D_{k-1}), 1 \leq k \leq n \}$$
\end{parag}

\begin{parag} \label{parag projections}
    The third class is given by the following projections:
    $$\p_n := \{p_2: N_n(J) \times \Theta_n[\theta] \to \Theta_n[\theta], \theta \in \Theta_n\}$$
\end{parag}

\begin{proposition} \label{recognition principle for left quillen functors}
Let $\m$ be a model category. Let $F: \tset \to \m$ be a cocontinuous functors sending monomorphisms to cofibrations. Then $F$ sends weak equivalences of $\nqcat$ to weak equivalences of $\m$ if and only if $F$ sends the morphisms in $\i_n \cup \j_n \cup \p_n$ to weak equivalences.
\end{proposition}

\begin{proof}
Suppose that $F$ sends weak equivalences to weak equivalences. By definition, the morphisms of $\i_n$ and $\j_n$ are weak equivalences of $\nqcat$, and thus they are sent to weak equivalences. Since $N_n(J) \to N_n(D_0)$ is a trivial fibration \cite[Corollary 6.7]{ara2014higher}, then for every $\theta \in \Theta_n$, the projection $N_n(J) \times \Theta_n[\theta] \to \Theta_n[\theta]$ is also a trivial fibration (as a pullback of a trivial fibration), and hence a weak equivalence. Therefore, the morphisms of $\p_n$ are also sent to weak equivalences.

To show the converse implication, we use the theory of localizers developed by Cisinski in \cite{cisinski2006prefaisceaux}. Let $\W$ be the class of morphisms of $\widehat{\Theta_n}$ which are sent to weak equivalences by $F$. We want to show that $\W(\i_n \cup \j_n) \subset \W$, where $\W(\i_n \cup \j_n)$ is the localizer generated by $\i_n \cup \j_n$, which is by definition the class of weak equivalences of $\nqcat$. For that, it is sufficient to show that $\W$ is a localizer, since it contains $\i_n \cup \j_n$ by hypothesis. This is done by applying \cite[Proposition 8.2.15]{cisinski2006prefaisceaux} to the \textit{pre-localizer} \cite[Def. 8.2.10]{cisinski2006prefaisceaux} $\W$ and the \textit{catégorie squelettique régulière} \cite[Definition 8.2.3]{cisinski2006prefaisceaux} $\Theta_n$. The hypothesis of the cited proposition is precisely the fact that $\p_n \subset \W$, since $N_n(J) \times \theta$ is a \textit{donnée homotopique élémentaire} (see Section \ref{section cylinder} for more details on the definition).
\end{proof}


\begin{parag} \label{parag simplicial presheaves}
    A \textit{simplicial presheaf} on a small category $\a$ is a functor $\a^{\op} \to \sset$. A simplicial presheaf $X: \a^{\op} \to \sset$ corresponds to a presheaf $\hat{X}$ on $\a \times \Delta$ by putting $\hat{X}_{(a, [n])} = X(a)_n$, for every $a \in \a$ and $n \geq 0$. From now on, we identify the categories of simplicial presheaves on $\a$ and of presheaves on $\a \times \Delta$, which we denote by $\spsh(\a)$.

    A simplicial presheaf $X$ on $\a$ is said to be \textit{discrete} if for every $a \in \a$, the simplicial set $X(a)$ is discrete. If $a$ is an object of $\a$, we denote by $F(a)$ the discrete simplicial presheaf whose value in $b \in \a$ is the constant simplicial set at $\Hom_{\a}(b, a)$.
\end{parag}

\begin{parag} \label{thetan spaces}
    In \cite{rezk2010cartesian}, Rezk defines a model structure $\thetansp$ on the category of simplicial presheaves on $\Theta_n$ (or presheaves on $\Theta_n \times \Delta$), which provides another model for $(\infty, n)$-categories. The fibrant objects of $\thetansp$ are called \textit{$(\infty,n)$-$\Theta$-spaces}.
\end{parag}





Let us recall the two equivalences proven in \cite{ara2014higher} between $n$-quasi-categories and $(\infty,n)$-$\Theta$-spaces. We also state some lemmas relating the introduced functors, which will be useful in the following sections.

\begin{parag} \label{parag p and i_0}
Let $p: \Theta_n \times \Delta \to \Theta_n$ be the projection functor and $i_0: \Theta_n \to \Theta_n \times \Delta$ be the functor that sends $\theta \in \Theta_n$ to $(\theta, [0]) \in \Theta_n \times \Delta$. Since $[0]$ is the terminal object of $\Delta$, there is an adjunction:
$$p: \Theta_n \times \Delta \rightleftarrows \Theta_n: i_0$$
which induces an adjunction between that presheaf categories 
\begin{equation} \label{adj ara rezk 1}
    p^*: \tset \rightleftarrows \stset: i_0^*
\end{equation}

\end{parag}

\begin{theorem} \label{quillen eq 1 ara rezk}
The adjunction (\ref{adj ara rezk 1}) is a Quillen equivalence 
\begin{equation*}
    p^*: \nqcat \rightleftarrows \thetansp: i_0^*
\end{equation*}
\end{theorem}

\begin{proof}
See \cite[Theorem 8.4.(1)]{ara2014higher}. 
\end{proof}

\begin{lemma} \label{lemma computation f e p^*}
    The functor $F: \Theta_n \to \stset$ defined in §\ref{parag simplicial presheaves} can be factored in two ways displayed in the following commuting (up to isomorphism) diagram
\[\begin{tikzcd}
	{\Theta_n} & {\Theta_n \times \Delta} \\
	\tset & \stset
	\arrow[from=1-1, to=2-1]
	\arrow["F", from=1-1, to=2-2]
	\arrow["{p^*}"', from=2-1, to=2-2]
	\arrow["{i_0}", from=1-1, to=1-2]
	\arrow[from=1-2, to=2-2]
\end{tikzcd}\]
where the vertical arrows are Yoneda embeddings.
\end{lemma}

\begin{proof}
    Straightforward computation from the definitions.
\end{proof}

\begin{parag}
    Let $G: \Delta \to \tset$ be the composite of the functors 
    $$\Delta \xrightarrow[]{} \Cat \xrightarrow[]{\Pi} \Gpd \to \ncat \xrightarrow[]{N_n} \tset $$
    where the unlabeled arrows are inclusions, $\Pi$ is the free-groupoid functor and $N_n$ is the strict $n$-nerve of §\ref{parag nerve}. We define $t: \Theta_n \times \Delta \to \tset$ on objects by $t(\theta, [n]) = \Theta_n[\theta] \times G([n])$. The functor $t$ induces an adjunction
    \begin{equation} \label{adj ara rezk 2}
        t_!: \stset \rightleftarrows \tset: t^*
    \end{equation}
\end{parag}

\begin{theorem} \label{quillen eq 2 ara rezk}
The adjunction (\ref{adj ara rezk 2}) is a Quillen equivalence 
\begin{equation*}
    t_!: \thetansp \rightleftarrows \nqcat : t^*
\end{equation*}
\end{theorem}

\begin{proof}
See \cite[Corollary 8.8]{ara2014higher}. 
\end{proof}

\begin{lemma} \label{lemma computation t_!}
    The following triangle commutes up to isomorphism of functors
\[\begin{tikzcd}
	& {\Theta_n} \\
	\stset && \tset
	\arrow["F"', from=1-2, to=2-1]
	\arrow[from=1-2, to=2-3]
	\arrow["{t_!}", from=2-1, to=2-3]
\end{tikzcd}\]
where the right diagonal functor is the Yoneda embedding.
\end{lemma}

\begin{proof}
    We want to show that the composite $t_! F$ is isomorphic to the Yoneda embedding of $\Theta_n$. Using the upper commutative triangle of Lemma \ref{lemma computation f e p^*}, we replace $F$ by $ \y i_0$, where $\y$ is the Yoneda embedding of $\Theta_n \times \Delta$. So $t_! F  \cong t_! \y i_0 = t i_0$, the last equality coming from the definition of $t_!$ as the left Kan extension of $t$ along $\y$. For every $\theta \in \Theta_n$, we have
    $$t i_0(\theta) = t(\theta, [0]) = \Theta_n[\theta] \times G([0]) \cong \Theta_n[\theta]$$
    since $G(0) = \Theta_n[0]$ is the terminal object of $\tset$.
\end{proof}



\section{A cylinder for $n$-quasi-categories} \label{section cylinder}

In this section, we recall the definition of a functorial cylinder in a presheaf category and of an elementary homotopical datum (\textit{donnée homotopique élémentaire}), following \cite{cisinski2006prefaisceaux}. We then specialize to the case of presheaves on a wreath product $\cc = \Delta \wr \a$, where $\a$ has a terminal object (denoted by $*$). We give a description, for every object $c$ of $\cc$, of the product $c \times [1](*)$ in $\psh(\cc)$ as a colimit of representables, which generalizes the description of products $[p] \times [1]$ of simplicial sets. Finally, we look at the case $\cc = \Theta_n (= \Delta \wr \Theta_{n-1})$, and modify this cylinder to obtain another one, better suited to the model structure for $n$-quasi-categories, while still having a simple combinatorial description.

Let $\cc$ be a small category.

\begin{definition} \label{definition cylinder}
A \textit{cylinder} for a presheaf $X$ on $\cc$ is the data of a presheaf $\i X$ and of morphisms $\partial_X^\varepsilon: X \to \i X$, $\varepsilon = 0,1$ and $\sigma_X: \i X \to X$ such that the following diagram commutes
\[\begin{tikzcd}
	X \\
	& \i X & X \\
	X
	\arrow["{\sigma_X}", from=2-2, to=2-3]
	\arrow["{\partial_X^0}", from=1-1, to=2-2]
	\arrow["{\partial_X^1}"', from=3-1, to=2-2]
	\arrow["{\id_X}"', curve={height=12pt}, from=3-1, to=2-3]
	\arrow["{\id_X}", curve={height=-12pt}, from=1-1, to=2-3]
\end{tikzcd}\]
and that $\langle \partial_X^0, \partial_X^1 \rangle : X \sqcup X \to \i X$ is a monomorphism.
\end{definition}

\begin{definition} \label{functorial cylinder}
    A \textit{functorial cylinder} on $\cc$ is the data of a functor $\i : \psh(\cc) \to \psh(\cc)$ and of natural transformations $\partial^\varepsilon: \id \to \i $, $\varepsilon = 0,1$, and $\sigma: \i  \to \id$, such that for every presheaf $X$ on $\cc$, $(\i (X), \partial^0_X, \partial^1_X, \sigma_X)$ is a cylinder for $X$.
\end{definition}

\begin{definition} \label{def DHE}
    An \textit{elementary homotopical datum} (EHD) on $\cc$ is a functorial cylinder $(\i , \partial^0, \partial^1, \sigma)$ verifying the two following axioms:

    (HD1) The functor $\i $ commutes with small colimits and preserves monomorphisms.
    
    (HD2) For every monomorphism $j: K \to L$ in $\psh(\cc)$ and $\varepsilon = 0,1$, the following square is a pullback
\[\begin{tikzcd}
	K & L \\
	{\i (K)} & {\i (L)}
	\arrow["j", from=1-1, to=1-2]
	\arrow["{\partial^\varepsilon_K}"', from=1-1, to=2-1]
	\arrow["{\partial^\varepsilon_L}", from=1-2, to=2-2]
	\arrow["{\i (j)}"', from=2-1, to=2-2]
\end{tikzcd}\]
\end{definition}

\begin{definition} \label{def interval}
    An \textit{interval} of $\psh(\cc)$ is a presheaf $I$ on $\cc$ equipped with two morphisms $\{\varepsilon\}: \{*\} \to I$, $\varepsilon = 0,1$. An interval is \textit{separating} if $\{0\} \cap \{1\} = \emptyset$.
\end{definition}

\begin{parag} \label{cylinder with separating interval}
If $I$ is a separating interval of $\psh(\cc)$, the functor $I \times (-): \psh(\cc) \to \psh(\cc)$ defines a functorial cylinder on $\cc$. The components at a presheaf $X$ of the natural transformations $\partial^\varepsilon$ and $\sigma$ are given by $\partial^\varepsilon_X =  X \times \{\varepsilon\} : X \cong X \times \{*\} \to X \times I$ and $\sigma_X = p_1: X \times I \to X$. This functorial cylinder is an EHD \cite[Example 1.3.8]{cisinski2006prefaisceaux}.
\end{parag}

\begin{parag}
    If $(I, \{\varepsilon_I\})$ is a separating interval and $j: I \to J$ is a monomorphism, then $J$ can be endowed with the structure of a separating interval by taking $\{ \varepsilon_J\} = j \{\varepsilon_I\}$.
\end{parag}

\begin{example} \label{example interval}
The nerves $N_n(D_1)$ and $N_n(J)$ are separating intervals of $\tset$.
\end{example}

Suppose that the category $\cc$ has a terminal object $[0]$. If $X$ a presheaf on $\cc$, there is a monomorphism $i_X: X_0 \to X$, where $X_0$ is seen as the constant presheaf at the set $X_0$ (simpler notation for $X_{[0]}$).

Our first goal is to show that, given a separating interval $I$ and a monomorphism $j: I \to J$, the presheaf
$$\j(X) = X \times I \bigsqcup_{X_0 \times I} X_0 \times J$$
has the structure of an EHD. 

\begin{parag} \label{j(X) as subpresheaf of X x J}
    Since both morphisms $i_X: X_0 \to X$ and $j: I \to J$ are monomorphisms, all the morphisms of the following pushout square are monomorphisms.

\[\begin{tikzcd}
	{X_0 \times I} & {X \times I} \\
	{X_0 \times J} & {\j(X)}
	\arrow[from=1-1, to=2-1]
	\arrow[from=1-1, to=1-2]
	\arrow[from=2-1, to=2-2]
	\arrow[from=1-2, to=2-2]
\end{tikzcd}\]

Since the square is also a pullback, the induced arrow $\j(X) \to X \times J$ is also a monomorphism, and we write $\j(X) = X \times I \cup X_0 \times J \subset X \times J$.
\end{parag}

\begin{proposition} \label{j(X) is an EHD}
The construction $X \mapsto \j(X)$ above defines a functorial cylinder on $\cc$, which is moreover an EHD.
\end{proposition}

\begin{proof}
The functoriality of $\j$ comes from the functoriality of $X \mapsto X_0$, the functoriality of the product and the universal property of the pushout.

The morphisms $\partial^\varepsilon_X: X \to \j(X)$ are the composites 
$$X \xrightarrow[]{X \times \{\varepsilon\}} X \times I \to \j(X)$$
Since $I$ is a separating interval, $X \times I$ is a cylinder and $X \sqcup X \to X \times I$ is a monomorphism. Therefore, $\langle \partial^0_X, \partial^1_X \rangle: X \sqcup X \to \j(X)$ is a monomorphism, as the composite of the monomorphisms $X \sqcup X \to X \times I$ and $X \times I \to \j(X)$.

The morphism $\sigma_X: X \to \j(X)$ is the composite of the inclusion $\j(X) \subset X \times J$ with the projection $X \times J \to X$. For $\varepsilon = 0,1$, we have $\sigma_X \partial^\varepsilon_X = \id_X$, as can be seen in the commuting diagram
\[\begin{tikzcd}
	X && {X \times I} & {\j(X)} & {X \times J} \\
	&& X
	\arrow["{X \times \{\varepsilon\}}", from=1-1, to=1-3]
	\arrow[from=1-3, to=1-4]
	\arrow[from=1-4, to=1-5]
	\arrow["{p_1}", from=1-3, to=2-3]
	\arrow["{\id_X}"', from=1-1, to=2-3]
	\arrow["{p_1}", curve={height=-12pt}, from=1-5, to=2-3]
	\arrow["{\partial^\varepsilon_X}", curve={height=-24pt}, from=1-1, to=1-4]
	\arrow["{\sigma_X}", from=1-4, to=2-3]
\end{tikzcd}\]
where all the horizontal arrows are inclusions. Thus $\j$ is a functorial cylinder. 

Let us show that $(\j, \partial^0, \partial^1, \sigma)$ is an EHD. The functor $\j$ commutes with small colimits, since products and pushouts of presheaves commute with small colimits. Now let $X \to Y$ be a monomorphism between presheaves on $\cc$. In particular, the map $X_0 \to Y_0$ is a monomorphism. The induced arrow $\j(X) \to \j(Y)$ is simply the inclusion $X \times I \cup X_0 \times J \subset Y \times I \cup Y_0 \times J$, so $\j$ preserves monomorphisms. It remains to show (HD2), that is, that given a monomorphism $X \to Y$ and $\varepsilon = 0,1$, the following commuting square is a pullback
\[\begin{tikzcd}
	X & Y \\
	{\j(X)} & {\j(Y)}
	\arrow[from=2-1, to=2-2]
	\arrow["{\partial^\varepsilon_X}"', from=1-1, to=2-1]
	\arrow["{\partial^\varepsilon_Y}", from=1-2, to=2-2]
	\arrow[from=1-1, to=1-2]
\end{tikzcd}\]
This is clear since all arrows are monomorphisms of presheaves of sets. \end{proof}


Now we work in the case $\cc = \Delta \wr \a$, where $\a$ is a small category with a terminal object $*$. Let $a \in \a$ and $I = [1](*)$. We will give a presentation of $a \times I$ as a colimit of representables (see Proposition \ref{prop characterisation of cylinder}).

\begin{parag} \label{alternative construction of theta x I}
    Let $c \in \cc$. By definition, $c$ is of the form $[p](a_1, \ldots, a_p)$, where $p \geq 0$ and $a_1, \ldots, a_p \in \a$. Let $0 \leq i \leq p$. We define $c^i \in \cc$ as 
    $$c^i = [p+1](a_1, \ldots, a_{i}, *, a_{i+1}, \ldots, a_p)$$
    for $1 \leq i < p$ and 
    $$c^0 = [p+1](*, a_1, \ldots, a_p)$$
    $$c^p = [p+1](a_1, \ldots, a_p, *)$$

We introduce morphisms $\alpha^i: c \to c^i$, for $i = 0, \ldots, p-1$, induced by $(0, \ldots, i, i+2, \ldots, p+1): [p] \to [p+1]$,  and morphisms $\beta^i: c \to c^i$, for $i = 1, \ldots, p$, induced by $(0, \ldots, i-1, i+1, \ldots, p+1): [p] \to [p+1]$. We can picture the morphisms $\alpha^i$ and $\beta^i$ by the following plain and dotted arrows, respectively.

\[\begin{tikzcd}
	c && {{}} & {i-1} & i & {i+1} & {{}} \\
	\\
	{c^i} & {{}} & {i-1} & i && {i+1} & {i+2} & {{}}
	\arrow["{a_i}", no head, from=1-4, to=1-5]
	\arrow["{a_{i+1}}", no head, from=1-5, to=1-6]
	\arrow[dashed, no head, from=1-3, to=1-4]
	\arrow[dashed, no head, from=1-6, to=1-7]
	\arrow[dashed, no head, from=3-2, to=3-3]
	\arrow[dashed, no head, from=3-7, to=3-8]
	\arrow["{a_i}", no head, from=3-3, to=3-4]
	\arrow["{a_{i+1}}", no head, from=3-6, to=3-7]
	\arrow["{*}", no head, from=3-4, to=3-6]
	\arrow[shift left=2, maps to, from=1-4, to=3-3]
	\arrow[maps to, from=1-5, to=3-4]
	\arrow[shift right=2, maps to, from=1-6, to=3-7]
	\arrow[shift right=2, dotted, maps to, from=1-4, to=3-3]
	\arrow[shift left=2, dotted, maps to, from=1-6, to=3-7]
	\arrow[dotted, maps to, from=1-5, to=3-6]
	\arrow["{\alpha^i}"', shift right=2, from=1-1, to=3-1]
	\arrow["{\beta^i}", shift left=2, dotted, from=1-1, to=3-1]
\end{tikzcd}\]
\end{parag}

Let $\i(c)$ be the colimit of the following diagram in $\tset$
\[\begin{tikzcd}
	& c && c & \ldots & c \\
	{c^0} && {c^1} &&&& {c^p}
	\arrow["{\alpha^0}", from=1-2, to=2-1]
	\arrow["{\beta^1}"', from=1-2, to=2-3]
	\arrow["{\alpha^1}", from=1-4, to=2-3]
	\arrow["{\beta^p}"', from=1-6, to=2-7]
\end{tikzcd}\]

\begin{proposition} \label{prop characterisation of cylinder}
    Let $\a$ be a small category with a terminal object $*$. Let $\cc$ be the wreath product $\Delta \wr \a$ let $I$ be the object $[1](*)$ of $\cc$. For every $c \in \cc$, there is an isomorphism
    $$\i(c) \xrightarrow[]{\cong} c \times I$$
    in $\psh(\cc)$.
\end{proposition}

\begin{proof}
    A morphism $\i(c) \to c \times I$ corresponds to a morphism $\i(c) \to c$ and a morphism $\i(c) \to I$. For $0 \leq i \leq p$, let $\sigma^i: c^{i} \to c$ be the morphism induced by $\sigma^{i}: [p+1] \to [p]$ sending $i$ and $i+1$ to $i$. Since for every $i = 1, \ldots p$ the diagrams

\[\begin{tikzcd}
	& {c^{i-1}} \\
	c && c \\
	& {c^i}
	\arrow["{\alpha^{i-1}}", from=2-1, to=1-2]
	\arrow["{\beta^i}"', from=2-1, to=3-2]
	\arrow["{\sigma^{i-1}}", from=1-2, to=2-3]
	\arrow["{\sigma^i}"', from=3-2, to=2-3]
\end{tikzcd}\]
commute (both composites equal $\id_c$), the morphisms $\sigma^i$ induce a morphism $\sigma: \i(c) \to c$.

For $0 \leq i \leq p$, let $\gamma^{i}: c^i \to [1](*) = I$ be the morphism sending $i$ to $0$ and $i+1$ to $1$. Since for every $i = 1, \ldots p$ the diagrams

\[\begin{tikzcd}
	& {c^{i-1}} \\
	c && I \\
	& {c^i}
	\arrow["{\alpha^{i-1}}", from=2-1, to=1-2]
	\arrow["{\beta^i}"', from=2-1, to=3-2]
	\arrow["{\gamma^{i-1}}", from=1-2, to=2-3]
	\arrow["{\gamma^i}"', from=3-2, to=2-3]
\end{tikzcd}\]
commute, the morphisms $\gamma^i$ induce a morphism $\gamma: \i(c) \to I$.

It remains to show that the morphism $\langle \sigma, \gamma \rangle: \i(c) \to c \times I$ is an isomorphism. We will show that, for every $d = [q](b_1, \ldots, b_q) \in \cc$, the induced morphism
$$\Hom_{\psh(\cc)}(d, \i(c)) \to \Hom_{\psh(\cc)}(d, c \times I)$$
is a bijection.

Let $f: d \to c \times I$, $g = p_1 f: d \to c$ and $h = p_2 f: d \to I$. Since $I = [1](*)$, the map $h$ corresponds to a morphism $[q] \to [1]$ in $\Delta$, which we still denote by $h$. If $h$ sends all objects of $[q]$ to $0$ (resp. 1), consider $f': d \to \i(c)$ defined as the composite $d \xrightarrow[]{g} c \xrightarrow[]{[\delta^{p+1}]} c^p \to \i(c)$ (resp. $d \xrightarrow[]{g} c \xrightarrow[]{[\delta^{0}]} c^0 \to \i(c)$). It is clear that $\langle \sigma, \gamma \rangle f' = f$. Note that the construction of $f'$ is the unique possible one.

If not all objects of $[q]$ are sent to the same object of $[1]$, then there exists $1 \leq j_h \leq q$ such that $h(i) = 0$ for $i < j_h$ and $h(i) = 1$ for $i \geq j_h$. Let $k = g(j_h - 1)$. Consider $f': d \to \i(c)$ defined as the composite $d \xrightarrow{f''} c^{k} \to \i(c)$ where $f''$ is induced by $g$ in the following way: $f''(i) = g(i)$ for $i < j_h$ and $f''(i) = g(i) + 1$ for $i \geq j_h$ (the maps between $b_j$'s and $a_i$'s are those of $g$). We have $\sigma^{k} f'' = g$ and $\gamma^{k} f'' = h$, so $\langle \sigma, \gamma \rangle f' = f$, as desired. This time, the only choice we made was that of $k$. Indeed, we could have chosen any $k$ from $g(j_h -1)$ to $g(j_h)-1$, and produced $f'_k: d \to \i(c)$ factoring through $f''_k: d \to c^k$. Suppose that there are two possible consecutive choices $k$ and $k+1$. Following the constructions, we see that $f''_k = \alpha^k g$ and $f''_{k+1} = \beta^{k+1} g$, and therefore $f''_k$ and $f''_{k+1}$ define the same morphism $f': d \to \i(c)$.

\end{proof}

In the last part of this section, we restrict our attention to the case $\cc = \Theta_n$, for $n >0$ (or $\a = \Theta_{n-1}$, since $\Theta_n = \Delta \wr \Theta_{n-1}$). As before, let $I = [1]([0])$ and $J = N_n(J)$. Recall that $\j(\theta) = \theta \times I \cup \theta_0 \times J$. Let us provide an alternative construction $\j'(\theta)$ for $\j(\theta)$.

\begin{parag} \label{definition theta^{i}_j}
    Let $\theta = [p](\tau_1, \ldots, \tau_p) \in \Theta_n$ and $0 \leq i \leq p$. Let $\theta^{i}_J$ be the presheaf on $\Theta_n$ defined as the following pushout

\[\begin{tikzcd}
	I & {\theta^{i}} \\
	J & {\theta^i_J}
	\arrow["{\{i, i+1\}}", from=1-1, to=1-2]
	\arrow[from=1-1, to=2-1]
	\arrow["{\varphi^{i}}", from=1-2, to=2-2]
	\arrow["{\psi^{i}}"', from=2-1, to=2-2]
	\arrow["\lrcorner"{anchor=center, pos=0.125, rotate=180}, draw=none, from=2-2, to=1-1]
\end{tikzcd}\]
where the vertical arrow is the inclusion $I \to J$ and the top horizontal arrow sends the object $0$ (resp. $1$) of $I$ to the object $i$ (resp. $i+1$) of $\theta^i$. Since the arrows $I \to \theta^{i}$ and $I \to J$ are monomorphisms, the arrows $\varphi^{i}$ and $\psi^{i}$ are also monomorphisms, as pushouts of monomorphisms of presheaves.

Let $\j'(\theta)$ be the colimit of the following diagram in $\tset$
\[\begin{tikzcd}
	& \theta && \theta & \ldots & \theta \\
	{\theta^0} && {\theta^1} &&&& {\theta^p} \\
	{\theta^0_J} && {\theta^1_J} &&&& {\theta^p_J}
	\arrow["{\alpha^0}", from=1-2, to=2-1]
	\arrow["{\beta^1}"', from=1-2, to=2-3]
	\arrow["{\alpha^1}", from=1-4, to=2-3]
	\arrow["{\beta^p}"', from=1-6, to=2-7]
	\arrow["{\varphi^{0}}", from=2-1, to=3-1]
	\arrow["{\varphi^1}", from=2-3, to=3-3]
	\arrow["{\varphi^p}", from=2-7, to=3-7]
\end{tikzcd}\]
\end{parag}

\begin{proposition} \label{prop arternative description of j(theta)}
    For every $\theta \in \Theta_n$, there is an isomorphism
    $$\j(\theta) \cong \j'(\theta)$$
    in $\tset$.
\end{proposition}

\begin{proof}
It follows from the fact that colimits commute with colimits. Let us explicit all the diagrams and functors. If $\theta = [p](\tau_1, \ldots, \tau_p)$, let $W_p$ be the category
\[\begin{tikzcd}
	& {1'} && {2'} & \ldots & {p'} \\
	0 && 1 &&&& p
	\arrow[from=1-2, to=2-1]
	\arrow[from=1-2, to=2-3]
	\arrow[from=1-4, to=2-3]
	\arrow[from=1-6, to=2-7]
\end{tikzcd}\]
and $P$ be the category
\[\begin{tikzcd}
	a & b \\
	c
	\arrow[from=1-1, to=2-1]
	\arrow[from=1-1, to=1-2]
\end{tikzcd}\]

Recall that $\j(\theta)$ is defined as the pushout
\[\begin{tikzcd}
	{\coprod_{0 \leq i \leq p} I} & {\i (\theta)} \\
	{\coprod_{0 \leq i \leq p} J} & {\j(\theta)}
	\arrow[from=1-1, to=2-1]
	\arrow[from=1-1, to=1-2]
	\arrow[from=2-1, to=2-2]
	\arrow[from=1-2, to=2-2]
	\arrow["\lrcorner"{anchor=center, pos=0.125, rotate=180}, draw=none, from=2-2, to=1-1]
\end{tikzcd}\]

We observe that the 3 objects defining this pushout can be seen as colimits indexed by $W_p$. Indeed, the coproduct $\coprod_{0 \leq i \leq p} I$ (resp. $\coprod_{0 \leq i \leq p} J$) is the colimit of $F^{I}: W_p \to \tset$ (resp. $F^J$) sending $j'$ to $\emptyset$ for every $1 \leq j \leq p$ and $i$ to $I$ (resp. $J$), for every $0 \leq i \leq p$. Moreover, $\i(\theta)$ is defined as the colimit of a functor $F^{I}_\theta: W_p \to \tset$. Therefore, consider the functor
$$F: W_p \times P \to \tset$$
whose adjoint 
$$F': P \to [W_p, \tset]$$
sends $a$ to $F^{I}$, $b$ to $F^{I}_{\theta}$ and $c$ to $F^J$. We have that 
$$\j(\theta) = \colim_{W_p} (\colim_P F')$$

The adjoint $$F'': W_p \to [P, \tset]$$ is the functor that sends $j'$ to the diagram
\[\begin{tikzcd}
	\emptyset & \theta \\
	\emptyset
	\arrow[from=1-1, to=1-2]
	\arrow[from=1-1, to=2-1]
\end{tikzcd}\]
for every $1 \leq j \leq p$ and $i$ to the diagram
\[\begin{tikzcd}
	I & {\theta^{i}} \\
	J
	\arrow["{\{i, i+1\}}", from=1-1, to=1-2]
	\arrow[from=1-1, to=2-1]
\end{tikzcd}\]
for every $0 \leq i \leq p$. The images by $F''$ of the arrows of $W_p$ are given by the unique morphisms $\emptyset \to I$ and $\emptyset \to J$ and by the morphisms $\alpha^{i}, \beta^{i}: \theta \to \theta^{i}$. By definition, we have 
$$\j'(\theta) = \colim_{P} (\colim_{W_p} F'')$$
and hence
$$\j(\theta) \cong \j'(\theta)$$
\end{proof}

\begin{parag} \label{contruction of morphism j'(theta) to theta}
    Let us construct a morphism $\mu: \j'(\theta) \to \theta$. Let $0 \leq i \leq p$. Recall from the proof of Proposition \ref{prop characterisation of cylinder} the morphism $\sigma^{i}: \theta^{i} \to \theta$. It fits in the commutative diagram below, inducing an arrow $\mu^{i}: \theta^{i}_J \to \theta$.

\[\begin{tikzcd}
	I & {\theta^{i}} \\
	J & {\theta^{i}_J} \\
	{\{*\}} && \theta
	\arrow["{\sigma^{i}}", curve={height=-12pt}, from=1-2, to=3-3]
	\arrow[from=2-1, to=3-1]
	\arrow["{\{i\}}"', from=3-1, to=3-3]
	\arrow[from=1-1, to=2-1]
	\arrow["{\{i, i+1\}}", from=1-1, to=1-2]
	\arrow[from=1-2, to=2-2]
	\arrow[from=2-1, to=2-2]
	\arrow["\lrcorner"{anchor=center, pos=0.125, rotate=180}, draw=none, from=2-2, to=1-1]
	\arrow["{\mu^{i}}", dashed, from=2-2, to=3-3]
\end{tikzcd}\]

The morphisms $\mu^{i}$ assemble to form a morphism $\mu: \j'(\theta) \to \theta$. It follows from the constructions that $\mu$ fits in a commutative triangle 
\[\begin{tikzcd}
	{\j(\theta)} && {\j'(\theta)} \\
	& \theta
	\arrow["\varphi", from=1-1, to=1-3]
	\arrow["\sigma"', from=1-1, to=2-2]
	\arrow["\mu", from=1-3, to=2-2]
\end{tikzcd}\]
where $\varphi$ is the isomorphism of Proposition \ref{prop arternative description of j(theta)} and $\sigma$ is described in the proof of Proposition \ref{j(X) is an EHD}.
\end{parag}

We now study these cylinder constructions considering the model structure for $n$-quasi-categories on $\tset$.

\begin{proposition} \label{muˆi weak equivalence}
    For $0 \leq i \leq p$, the morphism $\mu^{i}: \theta^{i}_J \to \theta$ is a weak equivalence of $\nqcat$.
\end{proposition}

\begin{proof}
Consider the commutative diagram below

\[\begin{tikzcd}
	I & {I[\theta^{i}]} & {\theta^{i}} \\
	J & A & {\theta^{i}_J} \\
	{\{*\}} & {I[\theta]} & B \\
	&&& \theta
	\arrow[from=1-1, to=1-2]
	\arrow[from=1-1, to=2-1]
	\arrow["\sim"', from=2-1, to=3-1]
	\arrow[from=1-2, to=2-2]
	\arrow[from=2-1, to=2-2]
	\arrow[from=3-1, to=3-2]
	\arrow[from=2-2, to=3-2]
	\arrow["\lrcorner"{anchor=center, pos=0.125, rotate=180}, draw=none, from=2-2, to=1-1]
	\arrow["\lrcorner"{anchor=center, pos=0.125, rotate=180}, draw=none, from=3-2, to=2-1]
	\arrow["\sim", from=1-2, to=1-3]
	\arrow[from=1-3, to=2-3]
	\arrow[from=2-2, to=2-3]
	\arrow["\lrcorner"{anchor=center, pos=0.125, rotate=180}, draw=none, from=2-3, to=1-2]
	\arrow[from=3-2, to=3-3]
	\arrow[from=2-3, to=3-3]
	\arrow["\lrcorner"{anchor=center, pos=0.125, rotate=180}, draw=none, from=3-3, to=2-2]
	\arrow["{\sigma^i}", curve={height=-12pt}, from=1-3, to=4-4]
	\arrow["{\{i\}}"', curve={height=12pt}, from=3-1, to=4-4]
	\arrow[dashed, from=3-3, to=4-4]
	\arrow["{\mu^{i}}", from=2-3, to=4-4]
	\arrow["\sim"', from=3-2, to=4-4]
\end{tikzcd}\]
where the upper rectangle
\[\begin{tikzcd}
	\bullet & \bullet & \bullet \\
	\bullet & \bullet & \bullet
	\arrow[from=1-1, to=1-2]
	\arrow[from=1-2, to=1-3]
	\arrow[from=2-1, to=2-2]
	\arrow[from=2-2, to=2-3]
	\arrow[from=1-1, to=2-1]
	\arrow[from=1-2, to=2-2]
	\arrow[from=1-3, to=2-3]
\end{tikzcd}\]
is the pushout defining $\theta^{i}_J$, where we factorised the arrow $I \to \theta^{i}$ through the spine inclusion $I[\theta^{i}] \to \theta^{i}$, which is a weak equivalence of $\nqcat$ by definition. The morphisms $J \to \{*\}$ and $I[\theta] \to \theta$ are also weak equivalences by definition of $\nqcat$. The objects $A$ and $B$ are defined as the respective pushouts, and we will not need to compute them explicitly.

Since $I \to I[\theta^{i}]$ and $I \to J$ are monomorphisms (equivalently, cofibrations of $\nqcat$), the arrows $I[\theta^{i}] \to A$ and $J \to A$ are also cofibrations, since cofibrations are stable by pushout. Moreover, the arrow $A \to \theta^{i}_J$ is a trivial cofibration, as a pushout of the trivial cofibration $I[\theta^{i}] \to \theta^{i}$. 

Since $\nqcat$ is left proper, we have that that $A \to I[\theta]$ and $\theta^{i}_J \to B$ are weak equivalences. The 2-out-of-3 property of weak equivalences then implies that $I[\theta] \to B$ and $B \to \theta$ are weak equivalences, and thus that $\mu^{i}: \theta^{i}_J \to \theta$ is a weak equivalence, as desired.
\end{proof}

\begin{theorem} \label{mu is a weak equivalence}
    For every $\theta \in \Theta_n$, the morphism $\mu: \j'(\theta) \to \theta$ is a weak equivalence of $\nqcat$.
\end{theorem}

\begin{proof}
   The morphism $\mu: \j'(\theta) \to \theta$ is induced by a morphism of diagrams $\W_p \to \tset$. Indeed, $\j'(\theta)$ is a colimit of a functor $F^{J}_\theta: W_p \to \tset$ by definition. We can see $\theta$ as such a colimit by considering $F_\theta: W_p \to \tset$ sending all objects of $W_p$ to $\theta$ and all arrows to the identity of $\theta$. The morphism of diagrams $F^{J}_\theta \to F_\theta$ is induced by the identities $\theta \to \theta$ (in particular, weak equivalences) and by the arrows $\mu^{i}: \theta^{i}_J \to \theta$, which are weak equivalences by Proposition \ref{muˆi weak equivalence}. 

   In the diagrams given by $F^J_\theta$ and by $F_\theta$, all objects are cofibrant and all arrows are cofibrations, so their colimits are in fact homotopy colimits (a proof of this fact for diagrams of shape $W_p$ can be found in the proof of \cite[Proposition B.2.1]{curien2022rigidification}, under the first diagram appearing in the proof). Since homotopy colimits preserve weak equivalences of diagrams, the induced arrow $\mu: \j'(\theta) \to \theta$ is a weak equivalence.
\end{proof}

\begin{corollary} \label{sigma theta is a weak equivalence}
For every $\theta \in \Theta_n$, the morphism $\sigma: \j(\theta) \to \theta$ is a weak equivalence in $\nqcat$.
\end{corollary}

\begin{proof}
    Apply the 2-out-of-3 property of weak equivalences to the commutative triangle of §\ref{contruction of morphism j'(theta) to theta}.
\end{proof}

\begin{corollary} \label{sigma X is a weak equivalence}
    For every $X \in \tset$, the morphism $\sigma_X: \j(X) \to X$ is a weak equivalence in $\nqcat$.
\end{corollary}

\begin{proof}
    We apply \cite[Proposition 8.2.15]{cisinski2006prefaisceaux}, using the fact that $\Theta_n$ is a \textit{catégorie squelettique régulière} and that $\j(X)$ is an EHD (cf. Proposition \ref{j(X) is an EHD}). We take $\W$ as the class of weak equivalences of $\nqcat$, which is a (pre-)localizer, and Corollary \ref{sigma theta is a weak equivalence} is exactly the hypothesis that $\sigma: \j(\theta) \to \theta$ is in $\W$ for every $\theta \in \Theta_n$.
\end{proof}

Recall from Paragraphs \ref{parag spine inclusions} and \ref{parag equivalences} the definitions of the classes $\i_n$ and $\j_n$. The following theorem is a refinement of the recognition principle given in Proposition \ref{recognition principle for left quillen functors}.
 
\begin{theorem} \label{recognition principle for left quillen functors 2}
    Let $\m$ be a model category. Let $F: \tset \to \m$ be a cocontinuous functor sending monomorphisms to cofibrations. Then $F$ sends weak equivalences of $\nqcat$ to weak equivalences of $\m$ if and only if $F$ sends the morphisms in $\i_n \cup \j_n$ to weak equivalences.
\end{theorem}

\begin{proof}
    The direct implication is clear, since the morphisms in $\i_n$ and $\j_n$ are weak equivalences by definition.

    For the converse implication, we use once again \cite[Proposition 8.2.15]{cisinski2006prefaisceaux}. The proof follows exactly as the second paragraph of the proof of Proposition \ref{recognition principle for left quillen functors}, the only difference is that we exchange the EHD $J \times \theta$ by the EHD $\j(\theta)$. 
    
    Therefore, it suffices to show that $F$ sends the morphims $\sigma: \j(\theta) \to \theta$ to weak equivalences of $\m$. Since, by hypothesis, the functor $F$ preserves pushouts, and sends spine inclusions and the projection $J \to \{*\}$ to weak equivalences of $\m$, we can apply the same reasoning of the proof of Proposition \ref{muˆi weak equivalence} to show that $F(\mu^{i}): F(\theta^{i}_J) \to F(\theta)$ is a weak equivalence of $\m$. Thus, the equivalence of diagrams $F^{J}_\theta \to F_\theta$ which induces the weak equivalence $\mu: \j'(\theta) \to \theta$ of Theorem \ref{mu is a weak equivalence} provides an equivalence of diagrams $F(F^{J}_\theta) \to F(F_\theta)$. Since $F$ preserves homotopy colimits, the same argument of the proof Theorem \ref{mu is a weak equivalence} applies, showing that $F(\mu): F(\j'(\theta)) \to F(\theta)$ is a weak equivalence, and hence that $F(\sigma)$ is a weak equivalence.
\end{proof}

\section{Truncated objects}

In this section, we recall Rezk's definition of $(n+k, n)$-$\Theta$-spaces and we define $(n+k)$-truncated $n$-quasi-categories, generalizing the definitions of \cite{joyal2008notes} for $n = 1$ and \cite{campbell2020homotopy} for $(n = 2, k =0)$. It is shown in \cite{campbell2020truncated} and \cite{campbell2020homotopy} that the adjunctions (\ref{adj ara rezk 1}) and (\ref{adj ara rezk 2}) induce Quillen equivalences between $(n+k, n)$-$\Theta$-spaces and $(n+k)$-truncated $n$-quasi-categories, for $(n = 1, k \geq 0)$ and $(n = 2, k = 0)$, respectively. Using the same techniques, we extend this results to all $n \geq 1$, $k \geq 0$, constructing a model structure for  $(n+k)$-truncated $n$-quasi-categories and showing that it is Quillen equivalent to the one for $(n+k, n)$-$\Theta$-spaces via the same two adjunctions.

\begin{parag} \label{homotopy k-types}
    Let $m \geq 0$. A Kan complex $X$ is a \textit{homotopy $k$-type} if it has the right lifting property with respect to the boundary inclusions $\partial \Delta[m] \to \Delta[m]$, for all $m \geq k+2$. Equivalently, a homotopy $k$-type is a Kan complex such that for all $x \in X_0$, the homotopy groups $\pi_m(X, x)$ are trivial for every $m > k$.
\end{parag}

\begin{parag} \label{hom-kan complex of a quasi-category}
    If $X$ is a quasi-category, and $x, y \in X_0$, the Kan complex $\Hom_X(x, y)$ is defined as the pullback
\[\begin{tikzcd}
	{\Hom_X(x, y)} & {X^{\Delta[1]}} \\
	{\{(x,y)\}} & {X^{\partial \Delta[1]}}
	\arrow[from=2-1, to=2-2]
	\arrow[from=1-2, to=2-2]
	\arrow[from=1-1, to=2-1]
	\arrow[from=1-1, to=1-2]
	\arrow["\lrcorner"{anchor=center, pos=0.125}, draw=none, from=1-1, to=2-2]
\end{tikzcd}\]
where the right vertical arrow is induced by the boundary inclusion $\partial \Delta[1] \to \Delta[1]$.
\end{parag}

\begin{parag} \label{def 1-truncated}
Let $k \geq 0$. Recall from \cite[§26]{joyal2008notes} that a (1-)quasi-category $X$ is \textit{$(1+k)$-truncated} if for every $x, y \in X_0$, the Kan complex $\Hom_X(x,y)$ is a homotopy $k$-type.
\end{parag}

\begin{parag} \label{adj suspension}
    For every $n \geq 1$, the suspension functor $\sigma: \Theta_n \to \Theta_{n+1}$ induces an adjunction
    \begin{equation} \label{adj suspension eq}
        \Sigma: \psh(\Theta_n) \rightleftarrows  * \sqcup * / \psh(\Theta_{n+1}) : \Hom
    \end{equation}
    between presheaves on $\Theta_n$ and bipointed presheaves on $\Theta_{n+1}$. Indeed, composing $\sigma$ with the Yoneda embedding $\Theta_{n+1} \to \psh(\Theta_{n+1})$, we get a functor $\Theta_{n} \to \psh(\Theta_{n+1})$ which factorises via the forgetful functor $* \sqcup * / \psh(\Theta_{n+1}) \to \psh(\Theta_{n+1})$. The functor $\hat{\sigma}: \Theta_{n} \to * \sqcup * / \psh(\Theta_{n+1})$ sends $\theta \in \Theta_{n}$ to $(\Theta_{n+1}[\sigma(\theta)], 0, 1)$. The adjunction (\ref{adj suspension eq}) is the Kan extension-nerve adjunction induced by this functor, i.e., $\Sigma := \hat{\sigma}_{!}$ and $\Hom := \hat{\sigma}^!$. 
    
    Explicitly, for $X \in \tset$ and $\theta \in \Theta_{n-1}$, the value of $\Hom_X(x,y) := \Hom(X, x, y)$ at $\theta$ is given by the pullback
\[\begin{tikzcd}
	{\Hom_X(x, y)_\theta} & {X_{\sigma(\theta)}} \\
	{\{(x,y)\}} & {X_0 \times X_0}
	\arrow[from=2-1, to=2-2]
	\arrow[from=1-2, to=2-2]
	\arrow[from=1-1, to=2-1]
	\arrow[from=1-1, to=1-2]
	\arrow["\lrcorner"{anchor=center, pos=0.125}, draw=none, from=1-1, to=2-2]
\end{tikzcd}\]
where the right vertical arrow is induced by the two morphisms $[0] \to \sigma(\theta)$ in $\Theta_n$.
\end{parag}

\begin{proposition} \label{quillen adj suspension}
The adjunction (\ref{adj suspension eq}) is a Quillen adjunction
$$\Sigma: \nqcat \rightleftarrows * \sqcup * / \npqcat: \Hom$$
\end{proposition}

\begin{proof}
The functor $\Sigma$ is cocontinuous (as a left adjoint) and preserves monomorphisms. Therefore, it suffices by Theorem \ref{recognition principle for left quillen functors 2} to show that it sends the morphisms of the classes $\i_n$ and $\j_n$ to weak equivalences. The spine inclusions $i_\theta$ are sent to the spine inclusions $i_{\sigma(\theta)}$, and the morphisms $j_k$ of $\j_n$ are sent to the morphisms $j_{k+1}$ of $\j_{n+1}$, so we are done. 
\end{proof}

\begin{definition}  \label{def truncated n-qcat}
Let $n \geq 2$, $k \geq 0$. An $n$-quasi-category $X$ is \textit{$(n+k)$-truncated} if for every $x,y \in X_0$ the $(n-1)$-quasi-category $\Hom_X(x,y)$ is $(n-1+k)$-truncated.
\end{definition}

\begin{parag} \label{parag mapping spaces for rezk model}
    Let $Z \in \stset$, and $x, y$ be two points of the space $Z(0)$. Rezk defines the mapping object $M_Z(x,y)$ to be the simplicial presheaf on $\Theta_{n-1}$ whose space at $\theta \in \Theta_{n-1}$ is given by the pullback
\[\begin{tikzcd}
	{M_Z(x,y)(\theta)} & {Z(\sigma(\theta))} \\
	{\{(x,y)\}} & {Z(0) \times Z(0)}
	\arrow[from=1-2, to=2-2]
	\arrow[from=2-1, to=2-2]
	\arrow[from=1-1, to=2-1]
	\arrow[from=1-1, to=1-2]
	\arrow["\lrcorner"{anchor=center, pos=0.125}, draw=none, from=1-1, to=2-2]
\end{tikzcd}\]
If $Z$ is an $(\infty, n)$-$\Theta$-space, then $M_Z(x,y)$ is an $(\infty, n-1)$-$\Theta$-space \cite[Proposition 8.3]{rezk2010cartesian}.
\end{parag}

Mapping objects for set-valued and space-valued presheaves on $\Theta_n$ can be related via the functor $i_0^*: \stset \to \tset$.

\begin{lemma} \label{lemma comparing mapping objects}
    Let $Z \in \stset$, and $x, y$ be two points of the space $Z(0)$. There is a natural isomorphism
    $$i_0^*(M_Z(x,y)) \cong \Hom_{i_0^*(Z)}(x,y)$$
\end{lemma}

\begin{proof}
    This follows directly from the definitions of the mapping objects and from the fact that $i_0^*$ preserves pullbacks, as a right adjoint (or as a restriction functor between presheaf categories).
\end{proof}

\begin{parag} \label{parag truncated rezk objects}
    Let $k \geq 0$. A \textit{$(k,0)$-$\Theta$-space} is the same as a homotopy $k$-type. An \textit{$(n+k, n)$-$\Theta$-space} $Z$ is an $(\infty, n)$-$\Theta$-space such that for every pair $(x,y)$ of points of $Z(0)$ the mapping object $M_Z(x,y)$ is an $(n-1+k, n-1)$-$\Theta$-space.
\end{parag}

\begin{proposition} [Rezk] \label{proposition rezk model st truncated}
    Let $k \geq 0$. There is a localisation $\thetanspkt$ of the model structure $\thetansp$ of $\stset$ such that the fibrant objects are exactly $(n+k,n)$-$\Theta$-spaces.
\end{proposition}

\begin{proof}
    The model structure is defined in \cite[§11.4]{rezk2010cartesian} and the characterisation of the fibrant objects is given in \cite[Proposition 11.20]{rezk2010cartesian}.
\end{proof}

\begin{theorem} \label{theorem truncated}
    Let $k \geq 0$. There is a model structure $\nqcatkt$ on $\tset$ whose fibrant objects are $(n+k)$-truncated $n$-quasi-categories. This model structure is the localization of $\nqcat$ with respect to the boundary inclusion $\partial \Theta_n[\sigma^{n-1} [k+3]] \to \Theta_n[\sigma^{n-1} [k+3]]$. Moreover, the adjunctions (\ref{adj ara rezk 1}) and (\ref{adj ara rezk 2}) are Quillen equivalences
    \begin{equation*}
    p^*: \nqcatkt \rightleftarrows \thetanspkt: i_0^*
\end{equation*}
and
    \begin{equation*}
    t_!: \thetanspkt \rightleftarrows \nqcatkt : t^*
\end{equation*}
\end{theorem}

\begin{proof}
    First, let us show that $(n+k)$-truncated $n$-quasi-categories are exactly the objects of $\tset$ which are local with respect to $\partial \Theta_n[\sigma^{n-1} [k+3]] \to \Theta_n[\sigma^{n-1} [k+3]]$ in $\nqcat$. We reason by induction on $n$. For $n=1$, this is \cite[Proposition 3.23]{campbell2020truncated}. Now let $X$ be an $(n+1)$-quasi-category. The following assertions are equivalent:
    \begin{itemize}
        \item $X$ is $(n+1+k)$-truncated
        \item For every $x,y \in X_0$, the $n$-quasi-category $\Hom_X(x,y)$ is $(n+k)$-truncated (by Definition \ref{def truncated n-qcat})
        \item For every $x,y \in X_0$, the $n$-quasi-category $\Hom_X(x,y)$ is local with respect $\partial \Theta_n[\sigma^{n-1} [k+3]] \to \Theta_n[\sigma^{n-1} [k+3]]$ in $\nqcat$ (by induction on $n$)
        \item For every $x,y \in X_0$ the bipointed $(n+1)$-quasi-category $(X, x, y)$ is local with respect to $\Sigma (\partial \Theta_n[\sigma^{n-1} [k+3]] \to \Theta_n[\sigma^{n-1} [k+3]]) = (\partial \Theta_{n+1}[\sigma^{n} [k+3]] \to \Theta_{n+1}[\sigma^{n} [k+3]])$ in $* \sqcup * / \npqcat$ (by Proposition \ref{local fibrant objects and localisation})
        \item $X$ is local with respect to $\partial \Theta_{n+1}[\sigma^{n} [k+3]] \to \Theta_{n+1}[\sigma^{n} [k+3]]$ in $\npqcat$ (by Lemma \ref{lemma local objects slice category})
    \end{itemize}
    
    Given the characterisation of truncated objects as local objects with respect to a single morphism, the existence of the model structure comes from Theorem \ref{smith existence theorem}.

    We know that the adjunctions in question are Quillen equivalences between $\nqcat$ and $\thetansp$ (by Theorems \ref{quillen eq 1 ara rezk} and \ref{quillen eq 2 ara rezk}). To show that they remain Quillen equivalences after localisation to their truncated versions, we use Theorem \ref{recognition of equivalence after localisastion}. 
    
    For the first adjunction, we have to show that an $(\infty, n)$-$\Theta$-space $Z$ is an $(n+k, n)$-$\Theta$-space if and only if $i_0^*(Z)$ is an $(n+k)$-truncated $n$-quasi-category. Once again, we proceed by induction on $n$. For $n=1$, this is \cite[Proposition 5.8.(2)]{campbell2020truncated}. The induction step follows from applying the definitions of truncated objects and using the isomorphism of Lemma \ref{lemma comparing mapping objects}.

    For the second one, we have to show that an $n$-quasi-category $X$ is $(n+k)$-truncated if and only if $t_!(X)$ is an $(n+k, n)$-$\Theta$-space. Note that, for $X \in \tset$ and $\theta \in \Theta_n$, we have isomorphisms
    $$i_0^*(t^!(X))_{\theta} = t^! (X)_{(\theta, 0)} = \Hom_{\tset}(\Theta_n[\theta] \times G(0), X) \cong X_\theta$$
    natural in $\theta$, and hence an isomorphism $X \cong i_0^*(t^!(X))$. So $X$ is $(n+k)$-truncated if and only if $i_0^*(t^!(X))$ is $(n+k)$-truncated, which by the last paragraph is the same as saying that $t^!(X)$ is an $(n+k, n)$-$\Theta$-space.
\end{proof}

\section{Groupoidal objects}

In this section, we define groupoidal $n$-quasi-categories, generalizing the definitions of Kan complexes and of groupoidal 2-quasi-categories \cite{brittes2022groupoidal}. We show that these are the fibrant objects of a localisation of $\nqcat$, and that both Quillen equivalences between $\nqcat$ and $\thetansp$ still work in the level of groupoidal objects.

\begin{parag} \label{inclusion delta theta_n}
The inclusion $i: \Delta \to \Theta_n$ induces an adjunction between the presheaf categories:
\begin{equation} \label{adj delta theta_n}
    i_!: \sset \rightleftarrows \tset: i^*
\end{equation}
\end{parag}

\begin{proposition}  \label{quillen adj inclusion delta}
The adjunction (\ref{adj delta theta_n}) is a Quillen adjunction
$$i_!: \qcat  \rightleftarrows \nqcat: i^*$$
\end{proposition}

\begin{proof}
    We use the recognition criterion for left Quillen functors (Theorem \ref{recognition principle for left quillen functors 2}). The functor $i_!$ preserves colimits and monomorphisms. By definition, it sends $\Delta[p]$ to $\Theta_n[i[p]]$. Since it preserves colimits, it preserves spine inclusions.

    Recall that the functor $i$ admits a left adjoint $\pi: \Theta_n \to \Delta$, sending $\theta = [p](\theta_1, \ldots, \theta_p)$ to $[p]$. Thus, the left adjoint $i_!$ is isomorphic to the restriction functor $\pi^*$. A direct computation shows that we have an isomorphism of functors $N_n i \cong i_! N: \Cat \to \tset$, and so $i_! (N(J) \to D_0) = (N_n(J) \to D_0)$.

    We have shown that $i_!$ sends the classes $\i_1$ and $\j_1$ to the classes $\i_n$ and $\j_n$, which are in particular weak equivalences.
\end{proof}

\begin{definition} \label{def groupoidal n-q-cat}
    A quasi-category is \textit{groupoidal} if it is a Kan complex. Let $n \geq 2$. A $n$-quasi-category $X$ is \textit{groupoidal} if
    \begin{enumerate}
        \item $i^*(X)$ is a Kan complex
        \item For every $x, y \in X_0$, the $(n-1)$-quasi-category $\Hom_X(x,y)$ is groupoidal
    \end{enumerate}
    For short, we will call a groupoidal $n$-quasi-category an \textit{$n$-quasi-groupoid}.
\end{definition}

Recall from §\ref{parag simplicial presheaves} the functor $F: \Theta_n \to \spsh(\Theta_n)$.

\begin{proposition} [Rezk]
    There exists a model structure $\thetanspgp$ on $\stset$ given by localizing the model structure $\thetansp$ with respect to the set $$\{\{F(\tau^n_k): F(D_k) \to F(D_{k-1}), 1 \leq k \leq n\}\}$$
\end{proposition}

\begin{proof}
    The argument amounts to the 2-out-of-3 property of weak equivalences, used to relate the morphisms presented here with the ones used in Rezk's original definition.

    The set of morphisms used by Rezk in his definition \cite[§11.25]{rezk2010cartesian} of the model structure $\thetanspgp$ is another one, namely $\{V[1]^k(T_\# q), 0 \leq k < n\}$, where $T_\# q: T_\# F[1] \to T_\# E$ in the notation of \cite{rezk2010cartesian}. Since the morphisms $V[1]^k(T_\# p)$ for $ 0 \leq k < n$ (where $T_\# p: T_\# E \to T_\# F[0]$) are already weak equivalences in $\thetansp$ \cite[§11.6]{rezk2010cartesian}, we can replace (using 2-out-of-3) the first set by $\{V[1]^k(T_\# F(\tau^1_1)), 0 \leq k < n\}$, where $\tau^1_1: [1] \to [0]$ is the unique morphism $[1] \to [0]$ in $\Delta$. This is precisely the set $\{F(\tau^n_k): F(D_k) \to F(D_{k-1}), 1 \leq k \leq n\}\}$, which follows from a straightforward calculation using \cite[Proposition 4.2, Lemma 11.10]{rezk2010cartesian}.
\end{proof}

\begin{theorem} \label{model st groupoidal nqcat}
    There is a model structure $\nqcatgp$ on $\tset$ whose fibrant objects are $n$-quasi-groupoids. This model structure is the localization of $\nqcat$ with respect to the morphisms $\tau^n_k: \Theta_n[D_k] \to \Theta_n[D_{k-1}]$, for $1 \leq k \leq n$. Moreover, the adjunctions (\ref{adj ara rezk 1}) and (\ref{adj ara rezk 2}) are Quillen equivalences
    \begin{equation*}
    p^*: \nqcatgp \rightleftarrows \thetanspgp: i_0^*
\end{equation*}
and 
\begin{equation*}
    t_!: \thetanspgp \rightleftarrows \nqcatgp : t^*
\end{equation*}
\end{theorem}

\begin{proof}
First, let us show that an $n$-quasi-category $X$ is local with respect to $\{\tau^n_k, 1 \leq k \leq n\}$ if and only if $X$ is groupoidal. For $n = 1$ the result holds, since Kan complexes are exactly quasi-categories which are local with respcet to $\Delta[1] \to \Delta[0]$ in $\qcat$ (cf. \cite[Proposition 3.30]{campbell2020truncated}). 

Suppose it holds for $n = m$, and let us show it is true for $n = m+1$. Let $X$ be an $(m+1)$-quasi-category. The underlying quasi-category $i^*(X)$ is a Kan complex if and only if it is local with respect to $\tau^1_1$. By Propositions \ref{inclusion delta theta_n} and \ref{local fibrant objects and localisation}, it amounts to say that $X$ is local with respect to $i_!(\tau^1_1) = \tau^{m+1}_1$. 

Now let $x, y \in X_0$. By the induction hypothesis, $\Hom_X(x, y)$ is groupoidal if and only if it is local with respect to $\{\tau^m_k, 1 \leq k \leq m\}$. By Propositions \ref{quillen adj suspension} and \ref{local fibrant objects and localisation}, it is the case exactly when $(X, x, y)$ is local with respect to the set of bipointed morphisms 
$$\Sigma (\{ \tau^m_k, 1 \leq k \leq m\ \}) = \{ \tau^{m+1}_k, 2 \leq k \leq m+1\ \}$$
We conclude using Lemma \ref{lemma local objects slice category} that $\Hom_X(x, y)$ is groupoidal for all $x, y \in X_0$ if and only if $X$ is local with respect to $\{ \tau^{m+1}_k, 2 \leq k \leq m+1\ \}$ in $(m+1)$-$\qcat$. 

Combining conditions (1) and (2) of Definition \ref{def groupoidal n-q-cat}, we have the desired result.

The existence of the model structure follows directly from Theorem \ref{smith existence theorem}.

Finally, we show that the adjunctions are indeed Quillen equivalences. Since the two adjunctions are already Quillen equivalences between $\nqcat$ and $\thetansp$ by Theorems \ref{quillen eq 1 ara rezk} and \ref{quillen eq 2 ara rezk}, it is sufficient  by Theorem \ref{transfer of localisation} to show that the class of morphisms by which we localise the right-hand side model structures is the image by the left Quillen functors of the class by which we localise the left-hand side model structures. This is the case thanks to Lemmas \ref{lemma computation f e p^*} and \ref{lemma computation t_!}, as the morphisms used in the localisation $\nqcatgp$ (resp. $\thetanspgp$) are the image by the Yoneda embedding (resp. by the functor $F: \Theta_n \to \stset$) of the morphisms $\tau^n_k$ of $\Theta_n$.
\end{proof}

In the Kan-Quillen model structure on simplicial sets, every representable is contractible. The same is true in the model structure $\nqcatgp$ on $\tset$.

\begin{proposition} \label{reprsentable is contractible in nqcatgp}
Let $\theta \in \Theta_n$. The morphism $\Theta_n[\theta] \to \Theta_n[0]$ is a weak equivalence in $\nqcatgp$.
\end{proposition}

\begin{proof}
    We omit the notation $\Theta_n[-]$ for representables in this proof. 

    Let $\theta \in \Theta_n$. The spine $I[\theta]$ is a colimit of the form
\[\begin{tikzcd}
	{D_{i_1}} && {D_{i_2}} && \ldots && {D_{i_m}} \\
	& {D_{j_1}} && {D_{j_2}} && {D_{j_{m-1}}}
	\arrow[from=2-2, to=1-1]
	\arrow[from=2-2, to=1-3]
	\arrow[from=2-6, to=1-7]
	\arrow[from=2-4, to=1-3]
	\arrow[from=2-4, to=1-5]
	\arrow[from=2-6, to=1-5]
\end{tikzcd}\]
for some $i_1, \ldots, i_m, j_1, \ldots, j_{m-1}$ with $i_k > j_k < i_{k+1}$, $1 \leq k < m$ (cf. §\ref{parag spine inclusions}).

By definition of the model structure $\nqcatgp$, the morphisms
    $$D_n \to D_{n-1} \to \ldots \to D_1 \to D_0$$
    are all weak equivalences, so the diagram above is equivalent in $\nqcatgp$ to the diagram of same shape and only $D_0$'s and identities. The colimit of both diagram is actually a homotopy colimit (see proof of Theorem \ref{mu is a weak equivalence}), so the equivalence of diagrams induces a weak equivalence $I[\theta] \to D_0$.

Since the spine inclusion $I[\theta] \to \theta$ is a weak equivalence of $\nqcat$ (and hence of $\nqcatgp$), the 2-out-of-3 property of weak equivalences implies that $\theta \to D_0$ is a weak equivalence of $\nqcatgp$.
\end{proof}

\section{$n$-quasi-groupoids vs spaces}

In \cite{rezk2010cartesian}, Rezk shows that there is a Quillen equivalence between the model structure for groupoidal objects in $\stset$ and the model structure for Kan complexes in $\sset$. Combining this equivalence with both equivalences of Theorem \ref{model st groupoidal nqcat}, we get respectively a zigzag of Quillen equivalences and a direct Quillen equivalence between the model structure for groupoidal $n$-quasi-categories and the one for Kan complexes. The aim of this section is to provide another direct Quillen equivalence between these model structures, induced by the inclusion $i: \Delta \to \Theta_n$, which does not factors through $\stset$.

\begin{parag} \label{functors of Rezk}
Let $\ev_0: \stset \to \sset$ be the functor defined on objects by $X \mapsto X([0], -)$. It is the right adjoint of an adjoint pair
\begin{equation} \label{adjunction rezk}
    \ct: \sset \rightleftarrows \stset: \ev_0
\end{equation}
where $\ct$ is the constant functor, i.e., for $K \in \sset$, $\ct(K)_\theta = K$.
\end{parag}

\begin{parag} \label{functor Joyal}
Let $k: \Delta \to \sset$ be the composite
$$\Delta \hookrightarrow \Cat \xrightarrow{\Pi} \Gpd \hookrightarrow \Cat \xrightarrow{N} \sset$$
This functor induces a Kan extension-nerve adjunction
\begin{equation} \label{adjunction joyal}
    k_!: \sset \rightleftarrows \sset: k^!
\end{equation}
\end{parag}

\begin{remark} \label{remark k^! anc core}
The functor $k^!$ is related to the maximal sub-Kan complex functor $J$, whose definition we recall. Let $X$ be a quasi-category, $\ho(X)$ be its homotopy category and $J: \Cat \to \Gpd$ be the maximal sub-groupoid functor, which is the right adjoint of the inclusion $\Gpd \to \Cat$. The simplicial set $J(X)$ is the pullback
\[\begin{tikzcd}
	{J(X)} & X \\
	{N(J\ho(X))} & {N(\ho(X))}
	\arrow[from=1-2, to=2-2]
	\arrow[from=2-1, to=2-2]
	\arrow[from=1-1, to=2-1]
	\arrow[from=1-1, to=1-2]
	\arrow["\lrcorner"{anchor=center, pos=0.125}, draw=none, from=1-1, to=2-2]
\end{tikzcd}\]

The simplicial set $J(X)$ is a Kan complex for every quasi-category $X$. Indeed, it is the maximal Kan complex included in $X$, since $J$ defines a functor $J: \mathbf{QCat} \to \mathbf{KanCx}$, right adjoint to the inclusion $\mathbf{KanCx} \to \mathbf{QCat}$, where $\mathbf{KanCx}$ (resp. $\mathbf{Qcat}$) is the full subcategory of $\sset$ formed by Kan complexes (resp. quasi-categories) \cite[Theorem 4.19]{joyal2008thetheory}.

If $X$ is a quasi-category, the Quillen adjunction $k_!: \kancx \rightleftarrows \qcat: k^!$ \cite[Theorem 6.22]{joyal2008thetheory} implies that $k^!(X)$ is a Kan complex. Moreover, there is a morphism $k^!(X) \to X$ induced by the inclusion $k([p]) \to \Delta[p]$ for every $[p] \in \Delta$. Therefore, the arrow $k^!(X) \to X$ factors through $J(X) \to X$. The morphism $k^!(X) \to J(X)$ is in fact a trivial fibration \cite[Proposition 6.26]{joyal2008thetheory}.

\end{remark}

\begin{lemma} \label{lemma commuting square 1}
The following square commutes up to isomorphism.
\[\begin{tikzcd}
	{\tset} & {\stset} \\
	{\sset} & {\sset}
	\arrow["{i^*}"', from=1-1, to=2-1]
	\arrow["{t^!}", from=1-1, to=1-2]
	\arrow["{\ev_0}", from=1-2, to=2-2]
	\arrow["{k^!}"', from=2-1, to=2-2]
\end{tikzcd}\]
\end{lemma}

\begin{proof}
    This is a direct computation using the definitions of the functors. Indeed, let $X \in \tset$ and $[m] \in \Delta$. We have
    $$\ev_0 t^!(X)_m = t^!(X)_{([0], [m])} = \Hom_{\tset}(\Theta_n[0] \times G([m]), X) \cong \Hom_{\tset}(N_n \Pi([m]), X)$$
    and 
    $$k^! i^*(X)_m = \Hom_{\sset}(k([m]), i^*(X)) \cong \Hom_{\tset}(i_! N \Pi([m]), X) \cong \Hom_{\tset} (N_n \Pi([m]), X)$$
\end{proof}

\begin{proposition}[Rezk] \label{quillen eq rezk spaces} The adjunction (\ref{adjunction rezk}) is a Quillen equivalence
$$\ct: \kancx \rightleftarrows \thetanspgp: \ev_0$$
\end{proposition}

\begin{proof}
See \cite[Proposition 11.27.(1)]{rezk2010cartesian}.
\end{proof}

\begin{proposition}[Joyal] \label{quillen eq joyal}
    The adjunction (\ref{adjunction joyal}) is a Quillen equivalence
    $$k_!: \kancx \rightleftarrows \kancx: k^!$$
\end{proposition}

\begin{proof}
    The adjunction is Quillen as a composite of two Quillen adjunctions: $$k_!: \kancx \rightleftarrows \qcat: k^!$$ which is a Quillen adjunction by \cite[Theorem 6.22]{joyal2008thetheory}, and $$\id: \qcat \rightleftarrows \kancx: \id$$ given by the fact that $\kancx$ is a left Bousfield localisation of $\qcat$. 

    Therefore, we have to show that the derived adjunction
    $$\mathbb{L} k_!: \Ho(\kancx) \rightleftarrows \Ho(\kancx): \mathbb{R} k^!$$
    is an equivalence of categories. This is true since for every (cofibrant) object $X$ of $\sset$, there is a natural weak homotopy equivalence $X \to k_! X$ \cite[Theorem 6.22]{joyal2008thetheory}, and then a natural isomorphism of functors $\id \cong \mathbb{L} k_!: \Ho(\kancx) \to \Ho(\kancx)$.
\end{proof}

\begin{proposition} \label{quillen adjunction kan cx groupoidal nqcat}
The adjunction (\ref{adj delta theta_n}) is a Quillen adjunction
$$i_!: \kancx \rightleftarrows \nqcatgp: i^*$$
\end{proposition}

\begin{proof}
    The adjunction $i_!: \qcat \rightleftarrows \nqcatgp: i^*$ is Quillen as a composite of the Quillen adjunctions  
    $$i_!: \qcat \rightleftarrows \nqcat: i^*$$
    of Proposition \ref{quillen adj inclusion delta} and 
    $$\id: \nqcat \rightleftarrows \nqcatgp: \id$$
    given by left Bousfield localisation. 
    
    Since for every $n$-quasi-groupoid $X$, the simplicial set $i^*(X)$ is a Kan complex (i.e., a fibrant object in $\kancx$) by definition, and $\kancx$ is a localisation of $\qcat$, the adjunction
    $$i_!: \kancx \rightleftarrows \nqcatgp: i^*$$
    is Quillen by Theorem \ref{criterion for quillen adjunction after localisation}.
\end{proof}

\begin{theorem} \label{quillen eq kan cx groupoidal nqcat}
The adjunction (\ref{adj delta theta_n}) is a Quillen equivalence
$$i_!: \kancx \rightleftarrows \nqcatgp: i^*$$
\end{theorem}

\begin{proof}
    We use the 2-out-of-3 property of Quillen equivalences, applied to the diagram of Lemma \ref{lemma commuting square 1} with the following model structures
\[\begin{tikzcd}
	\nqcatgp & \thetanspgp \\
	\kancx & \kancx
	\arrow["{i^*}"', from=1-1, to=2-1]
	\arrow["{t^!}", from=1-1, to=1-2]
	\arrow["{\ev_0}", from=1-2, to=2-2]
	\arrow["{k^!}"', from=2-1, to=2-2]
\end{tikzcd}\]
    
    The horizontal arrows are the right Quillen functors of Quillen equivalences by Theorem \ref{model st groupoidal nqcat} and Proposition \ref{quillen eq joyal}. The right vertical arrow is the right Quillen functor of the Quillen equivalence of Proposition \ref{quillen eq rezk spaces}. The left vertical arrow is the right Quillen functor of a Quillen adjunction by Proposition \ref{quillen adjunction kan cx groupoidal nqcat}.
\end{proof}

\begin{parag} \label{weakly equivalent objects}
    We say that two objects of a model category are \textit{weakly equivalent} if there is a zigzag of weak equivalences between them, that is, if they are isomorphic in the homotopy category.
\end{parag}

We introduce some terminology to clarify the nature of the next lemma.

\begin{parag} \label{spheres in a model category}
    Let $\m$ be a model category with a terminal object $*$. We will define what we mean by a sphere in $\m$. The 0-sphere $\Sph^0$ is the object $* \sqcup *$. For $k >0$, the $k$-sphere $\Sph^k$ is the homotopy colimit of the diagram
\[\begin{tikzcd}
	{\Sph^{k-1}} & {*} \\
	{*}
	\arrow[from=1-1, to=1-2]
	\arrow[from=1-1, to=2-1]
\end{tikzcd}\]
in $\m$. This definition is unique up to isomorphism in the homotopy category $\Ho(\m)$, i.e., all $k$-spheres are weakly equivalent.
\end{parag}

\begin{example} \label{ex spheres in kan-quillen}
Consider the Kan-Quillen model structure on the category of simplicial sets. For every $k > 0$, the boundary $\partial \Delta[k]$ is a model for $\Sph^{k-1}$. Indeed, $\partial \Delta[1] = \Delta[0] \sqcup \Delta[0]$, and $\Delta[0]$ is the terminal object of $\sset$. By induction, suppose that $\partial \Delta[k]$ is a model for $\Sph^{k-1}$. Consider the following pushout diagram
\[\begin{tikzcd}
	{\partial \Delta[k]} & {\Lambda^0[k+1]} \\
	{\Delta[k]}
	\arrow[from=1-1, to=1-2]
	\arrow[from=1-1, to=2-1]
\end{tikzcd}\]
where the horizontal arrow sends $\partial \Delta[k]$ to the boundary of the missing face, and the vertical arrow is the boundary inclusion. It can be used to compute the homotopy pushout in the definition above (cf. \cite[Proposition A.2.4.4]{lurie2009higher}), since by induction $\partial \Delta[k]$ models $\Sph^{k-1}$, and we know that all objects are cofibrant, both arrows are cofibrations, and the objects $\Delta[k]$ and $\Lambda^0[k+1]$ are contractible in $\kancx$. The pushout, which equals $\partial \Delta[k+1]$, is then a model of $\Sph^{k}$.
\end{example}

In the following statement and proof, we omit the notation of the fully faithful inclusions $\Theta_m \to \Theta_n$, for $m < n$.

\begin{lemma} \label{sphere lemma}
    Let $k > 0$. For every $0 \leq \ell < n$ and $p \geq 0$ such that $\ell + p = k$, the object $\partial \Theta_n [\sigma^{\ell}[p]]$ of $\tset$ is a $(k-1)$-sphere in $\nqcatgp$.
\end{lemma}

\begin{proof} 
The proof follows by induction on $n$ and $k$.

For $n = 1$, we have $\Theta_1 = \Delta$, $\ell = 0$, and the result is explained in Example \ref{ex spheres in kan-quillen}. Now suppose the result is true for $\Theta_{n-1}$, for $n > 1$, and we will show that it is true for $\Theta_n$. 

We proceed by induction on $k$. For $k = 1$, the statement is true, since we have
$$\partial \Theta_n[\sigma [0]] = \partial \Theta_n[1] = \Theta_n[0] \sqcup \Theta_n[0]$$

Now suppose it is true for $k' \leq k$, $k > 0$. We want to show that $\partial \Theta_n [\sigma^{\ell}[p]]$ is a $k$-sphere, for $\ell + p = k+1$. 

For $\ell = 0$ (so $p = k+1$), we have 
$$\partial \Theta_n[k+1] = i_! (\partial \Delta[k+1])$$
which is a $k$-sphere, given the Example \ref{ex spheres in kan-quillen} and the fact that $i_!: \kancx \to \nqcatgp$ is a left Quillen functor, preserving both cofibrations and weak equivalences (since all objects are cofibrant) and sending the terminal object of $\sset$ to the terminal object of $\tset$.

For $\ell > 0$, we have
$$\partial \Theta_n [\sigma^{\ell} [p]] = \Sigma ( \partial \Theta_{n-1} [\sigma^{\ell - 1}[p]])$$
and $ \partial \Theta_{n-1} [\sigma^{\ell - 1}[p]]$ is a model for $\Sph^{k-1}$ by the induction hypothesis. If $p >0$, it can be presented as a homotopy pushout
\[\begin{tikzcd}
	{\partial \Theta_{n-1} [\sigma^{\ell - 1}[p-1]]} & {*} \\
	{*} & {\partial \Theta_{n-1} [\sigma^{\ell - 1}[p]]}
	\arrow[from=1-1, to=1-2]
	\arrow[from=1-1, to=2-1]
	\arrow[from=2-1, to=2-2]
	\arrow[from=1-2, to=2-2]
	\arrow["h"{description, pos=0.3}, "\lrcorner"{anchor=center, pos=0.125, rotate=180}, draw=none, from=2-2, to=1-1]
\end{tikzcd}\]

For every $n$, the Quillen adjunction $\Sigma: \nmqcat \rightleftarrows * \sqcup * / \nqcat: \Hom$ of Proposition \ref{quillen adj suspension} descends to a Quillen adjunction $\nmqcatgp \rightleftarrows * \sqcup * / \nqcatgp$ by Proposition \ref{criterion for quillen adjunction after localisation}, since if $X$ is a $n$-quasi-groupoid, then $\Hom_X(x,y)$ is a $(n-1)$-quasi-groupoid by definition.

Therefore, applying $\Sigma$ to the diagram above, we obtain a homotopy pushout square
\[\begin{tikzcd}
	{\partial \Theta_{n} [\sigma^{\ell}[p-1]]} & {D_1} \\
	{D_1} & {\partial \Theta_{n} [\sigma^{\ell}[p]]}
	\arrow[from=1-1, to=1-2]
	\arrow[from=1-1, to=2-1]
	\arrow[from=2-1, to=2-2]
	\arrow[from=1-2, to=2-2]
	\arrow["h"{description, pos=0.3}, "\lrcorner"{anchor=center, pos=0.125, rotate=180}, draw=none, from=2-2, to=1-1]
\end{tikzcd}\]
in $\nqcatgp$, and we conclude by noting that $\Theta_n[D_1] \to \Theta_n[D_0]$ is a weak equivalence in this model category by definition.

It remains only to treat the case where $\ell > 0$ and $p = 0$, and so $\ell = k+1 < n$. In this case, we have $\Theta_n[\sigma^{\ell}[0]] = \Theta_n[D_{k+1}]$. Consider the following pushout square
\[\begin{tikzcd}
	{\partial \Theta_{n} [D_k]} & {\Theta_n[D_k]} \\
	{\Theta_n[D_k]} & {\partial \Theta_{n} [D_{k+1}]}
	\arrow[from=1-1, to=1-2]
	\arrow[from=1-1, to=2-1]
	\arrow[from=2-1, to=2-2]
	\arrow[from=1-2, to=2-2]
	\arrow["\lrcorner"{anchor=center, pos=0.125, rotate=180}, draw=none, from=2-2, to=1-1]
\end{tikzcd}\]
Note that it is a homotopy pushout, that $\partial \Theta_{n} [D_k]$ is a $(k-1)$-sphere by induction and that $\Theta_n[D_k]$ is contractible by the definition of $\nqcatgp$. Therefore, $\partial \Theta_{n} [D_{k+1}]$ is a model for $\Sph^k$.
\end{proof}

\begin{parag} \label{parag model st for groupoidal and truncated}
    We denote by $\nqcatgpkt$ the localisation of $\nqcatgp$ whose fibrant objects are $(n+k)$-truncated $n$-quasi-groupoids. By Theorem \ref{theorem truncated}, it is the localisation of $\nqcatgp$ with respect to the boundary inclusion $\partial \Theta_n[\sigma^{n-1}[k+3]] \to \Theta_n[\sigma^{n-1}[k+3]]$.
\end{parag}

\begin{parag} \label{review of homotopy types}
Let $m \geq 0$. Homotopy $m$-types can be characterised as local objects in the Kan-Quillen model structure on simplicial sets. Indeed, a Kan complex $X$ is an $m$-type if and only if it is local with respect to the boundary inclusion $\partial \Delta[m+2] \to \Delta[m+2]$ in this model structure (see for example \cite[Corollary 3.25]{campbell2020truncated}). Therefore, by Theorem \ref{smith existence theorem}, the Bousfield localisation of the Kan-Quillen model structure with respect to $\delta_{m+2}: \partial \Delta[m+2] \to \Delta[m+2]$ produces a model structure whose fibrant objects are homotopy $m$-types, which we denote by $\kancxm$.
\end{parag}

\begin{theorem} \label{quillen eq truncated n-q-groupoids and homotopy n-types}
    The adjunction (\ref{adj delta theta_n}) is a Quillen equivalence
    $$i_!: \kancxnk \rightleftarrows \nqcatgpkt: i^*$$
\end{theorem}

\begin{proof}
    Starting from the Quillen equivalence
    $$i_!: \kancx \rightleftarrows \nqcatgp: i^*$$
    of Theorem \ref{quillen eq kan cx groupoidal nqcat}, we can localize $\kancx$ with respect to the morphism $\delta_{[n+k+2]}: \partial \Delta[n+k+2]  \to \Delta[n+k+2]$ and $\nqcatgp$ with respect to the image of this morphism by the (left-derived) functor $i_!$, i.e., by 
    $$i_!(\delta_{[n+k+2]}: \partial \Delta[n+k+2]  \to \Delta[n+k+2]) = (\delta_{[n+k+2]}: \partial \Theta_n[i[n+k+2]]  \to \Theta_n[i[n+k+2]])$$

    After localisation, we obtain a new Quillen equivalence
    $$i_! \kancxnk \rightleftarrows \nqcatgp_{\loc}: i^*$$
    by Theorem \ref{transfer of localisation}.

    We want to show that the localised model structure $\nqcatgp_{\loc}$ is exactly $\nqcatgpkt$. Recall from §\ref{parag model st for groupoidal and truncated} that $\nqcatgpkt$ is the localisation of $\nqcatgp$ with respect to the boundary inclusion $\delta_{\sigma^{n-1}[k+3]}: \partial \Theta_n[\sigma^{n-1}[k+3]] \to \Theta_n[\sigma^{n-1}[k+3]]$.

    Thanks to Lemma \ref{sphere lemma}, we know that both $\partial \Theta_n[i[n+k+2]]$ and $\partial \Theta_n[\sigma^{n-1}[k+3]]$ are models for the $(n+k+1)$-sphere in $\nqcat$, so they are weakly equivalent in this model structure. Moreover, by Proposition \ref{reprsentable is contractible in nqcatgp}, all representable $\Theta_n$-sets are contractible in $\nqcatgp$. We can then draw the following commutative diagram
\[\begin{tikzcd}
	{\partial \Theta_n[i[n+k+2]]} & \bullet & \ldots & \bullet & {\partial \Theta_n[\sigma^{n-1}[k+3]]} \\
	{\Theta_n[i[n+k+2]]} && {\Theta_n[0]} && {\Theta_n[\sigma^{n-1}[k+3]]}
	\arrow["{\delta_{[n+k+2]}}"', from=1-1, to=2-1]
	\arrow["{\delta_{\sigma^{n-1}[k+3]}}", from=1-5, to=2-5]
	\arrow["\sim", from=1-1, to=1-2]
	\arrow["\sim"', from=1-3, to=1-2]
	\arrow["\sim"', from=1-5, to=1-4]
	\arrow["\sim", from=1-3, to=1-4]
	\arrow["\sim"', from=2-1, to=2-3]
	\arrow["\sim", from=2-5, to=2-3]
	\arrow[from=1-2, to=2-3]
	\arrow[from=1-4, to=2-3]
	\arrow[from=1-1, to=2-3]
	\arrow[from=1-5, to=2-3]
\end{tikzcd}\]
    where the arrows marked with $\sim$ are weak equivalences in $\nqcatgp$ and the bullets represent some objects appearing in the zigzag of weak equivalences.

    In $\nqcatgp_{\loc}$ (resp. $\nqcatgpkt$), the left (resp. right) vertical arrow is a weak equivalence, so the right (resp. left) vertical arrow is also a weak equivalence, by 2-out-of-3. Since localising a model structure with respect to a weak equivalence does not change anything, and since localisations commute (cf. Proposition \ref{proposition succesive localisations}), we have:
    $$\nqcatgp_{\loc} = L_{\delta_{\sigma^{n-1}[k+3]}} \nqcatgp_{\loc} = L_{\delta_{[n+k+2]}} \nqcatgpkt = \nqcatgpkt$$
\end{proof}

We have seen that, working with different model structures on the category of $\Theta_n$-sets, we are able to obtain models for the theory of $(\infty, n)$-categories (with Ara's model structure), of $(n+k, n)$-categories (with the truncated model structures) and of $(n+k, 0)$-categories (with the groupoidal model structures) for $k \geq 0$. We actually expect that it is possible to model any theory of $(m+k, m)$-categories inside the realm of $\Theta_n$-sets, for $m \leq n$ and $0 \leq k \leq \infty$. For example, the passage from $(\infty, n)$ to $(\infty, m)$ should be made by localizing by all morphisms $D_{j} \to D_{j-1}$, for $m < j \leq n$. The case $n = 2$, $m = 1$ is proven by Campbell in \cite[Theorem 11.14]{campbell2020homotopy}. However, for the general case one should develop other techniques, that are not within the scope of this work.

\bibliographystyle{alpha}
\bibliography{sample}

\end{document}